\documentclass{amsart}
\usepackage{amsmath,amsfonts,amsthm,amssymb,amsrefs}
\usepackage{tikz}
\usetikzlibrary{arrows}

\newtheorem{theorem}{Theorem}
\newtheorem{cor}{Corollary}
\newtheorem{lemma}{Lemma}
\newtheorem{prop}{Proposition}

\begin{document}
\title{Flat grafting deformations of quadratic differentials on surfaces}
\date{}
\author{Ser-Wei Fu}

\begin{abstract}
In this paper we introduce flat grafting as a deformation of quadratic differentials on a surface of finite type that is analogous to the grafting map on hyperbolic surfaces. Flat grafting maps are generic in the strata structure and preserve parallel measured foliations. We use flat grafting to construct paths connecting any pair of quadratic differentials. In other words, we characterize cone point splitting deformations. The slices of quadratic differentials closed under flat grafting maps with a fixed direction arise naturally and we prove rigidity properties with respect to the lengths of closed curves.
\end{abstract}

\maketitle


\section{Introduction}\label{INTRO}

Let $S = S_{g,p}$ be an orientable compact surface of genus $g$ with $p$ marked points (or punctures). The moduli space $\mathcal{M}(S)$ has been studied extensively both in the Riemann surface setting and the hyperbolic geometry setting. In this paper we are interested in the space of quadratic differentials $QD(S)$ that is understood as the cotangent bundle over the Teichm\"{u}ller space $\mathcal{T}(S)$. Each element of $QD(S)$ induces a \emph{singular flat metric} on $S$, see \cites{Str,Zor}. We use $QD_1(S)$ to denote unit-area marked quadratic differentials.

Grafting along a simple closed geodesic on a hyperbolic surface is the process of cutting along the geodesic and gluing the boundary of a flat cylinder to the two sides of the geodesic. The grafting map produces a new conformal structure, and hence a self-homeomorphism of the Teichm\"{u}ller space. The operation is classical and more information can be found in \cite{KamTan}. Flat grafting is an analogous process that grafts along a simple closed geodesic with respect to a quadratic differential. See Section~\ref{EXAMP} and \ref{CONST} for examples and the definition of flat grafting. Flat grafting can also be seen as a global version of the parallelogram construction in \cites{EMZ,MZ,Boi}, see the remark at the end of Section~\ref{CONST}.

Each flat grafting deformation requires a choice of a simple closed curve and a vector $u\in\mathbb{R}^2$ that describes the direction and width. Restricting to unit-area quadratic differentials, we normalize the flat grafting deformation and prove the following properties in Section~\ref{PROPE}.

\begin{theorem}\label{mainprop}
Let $\overline{F}_u : QD_1(S) \times \mathcal{S}(S) \to QD_1(S)$ be a normalized flat grafting map with $u \neq (0,0)$.
\begin{itemize}
\item[(A)] The measured foliation class induced by the straight line foliation parallel to $u$ is preserved under the map.
\item[(B)] $\overline{F}_u(q,\alpha) = q$ if and only $\alpha$ is projectively equivalent to the straight line foliation perpendicular to $u$.
\item[(C)] $\overline{F}_u(q, \ \cdot\ ): \mathcal{S}(S) \to QD_1(S)$ is injective.
\item[(D)] $\overline{F}_u(\ \cdot\ ,\alpha): QD_1(S) \to QD_1(S)$ is continuous at $q$ if and only if geodesic representatives of $\alpha$ with respect to $q$ contain no saddle connections parallel to $u$.
\item[(E)] $\overline{F}_\cdot(q,\alpha): \mathbb{R}^2 \to QD_1(S)$ is continuous at $u$ if and only if geodesic representatives of $\alpha$ with respect to $q$ contain no saddle connections parallel to $u$.
\end{itemize}
\end{theorem}

The space of quadratic differentials $QD(S)$ is naturally stratified by the cone angles of the singularities of the induced singular flat metric. We analyze how the strata structure changes under the flat grafting deformation in Section~\ref{STRAT}. The principal stratum is the largest stratum (open and dense in $QD(S)$) where all singularities have cone angle $3\pi$ or $\pi$ at non-marked or marked points, respectively.  

\begin{theorem}\label{principal}
For all $q \in QD(S)$ and all $u\neq (0,0)$, there exists a simple closed curve $\alpha$ such that $F_u(q,\alpha)$ lies in the principal stratum.
\end{theorem}

As a corollary of the proof of Theorem~\ref{principal}, we can construct paths between any pair of quadratic differentials. From a different point of view, flat grafting can be used to describe open neighborhoods of quadratic differentials that span over multiple strata. The proof of Theorem~\ref{principal} illustrates how the combinatorics of the simple closed curves encode the deformations of the relative positioning of the cone points. The splitting and merging of cone points as the strata structure changes are the main features captured by the flat grafting deformation.

Consider a \emph{slice} of the space of quadratic differentials 
\[
QD^\mu(S) = \{ q \in QD(S) \mid \mu \text{ is the horizontal foliation of } q\}.
\]
The slices are preserved under any horizontal flat grafting deformation. In Section~\ref{RIGID} we prove a length rigidity result that indicates how a finite set of closed curves can determine the quadratic differential with the corresponding length vector. A family of metrics is \emph{length rigid} with respect to a set of curves if the length vector function is injective.

\begin{theorem}\label{rigidity}
Any compact set inside $QD_1^\mu(S)$ is length rigid with respect to a finite set of closed curves.
\end{theorem}

A slice $QD_1^\mu(S)$ can be seen as a family of Teichm\"{u}ller geodesics with a common ``direction". The horizontal flat grafting deformation along $\alpha$ is an automorphism of $QD_1^\mu(S)$ that pulls $\alpha$ straight and shortens the flat length of $\alpha$. It would be interesting to consider the dynamics of horizontal flat grafting as well as its interplay with the Teichm\"{u}ller geodesic flow and the horocyclic flow. An equally important viewpoint is to consider the flat grafting rays inside the moduli space for quadratic differentials or the induced ray in the moduli space, but that is beyond the scope of this paper. 

This paper is motivated by the question: are compact sets of flat surfaces (induced by $QD(S)$) length rigid with respect to a finite set of closed curves? As part of the Author's thesis work \cite{Fu}, the question is answered for compact sets within a single stratum. The obstruction coming from the strata structure is illustrated in the choices involved in splitting a cone point into multiple cone points. Flat grafting is constructed in order to simultaneously split cone points and maintain control over lengths of curves. The construction and the properties are aimed at providing analogies with the earthquake deformation in \cites{Ker,Thu,Bon} and the grafting operation in \cites{DumWol,ScaWol}. 

\noindent \textbf{Acknowledgement.} The author thanks Chris Leininger, Spencer Dowdall, Howard Masur, Jayadev Athreya, Matthew Stover, and many other people for helpful discussions over the development of the concept of flat grafting. The author would also like to thank UIUC, Temple University, and NTU for their excellent research environment.


\section{Quadratic differentials and flat metrics}\label{BACKG}

A \emph{quadratic differential} on a Riemann surface is a local structure of the form $\varphi dz^2$, where $\varphi$ is a holomorphic function. On each local chart, there is a well-defined vertical and horizontal direction that are respected by the transition maps. Furthermore, by integrating a choice of square root of $\varphi$ we can obtain an Euclidean structure with the transition maps being semi-translations. An important class of quadratic differentials comes from squaring an abelian differential, see \cite{Str}.

We use $QD(S)$ to denote the space of marked quadratic differentials and $QD_1(S)$ to denote the space of unit-area marked quadratic differentials. By forgetting the vertical direction, we obtain $\mathsf{Flat}(S)$, the space of Euclidean cone metrics induced by quadratic differentials. Quadratic differentials, abelian differentials, and the induced metric appear in many contexts and are studied in many forms, see \cites{AEZ,EMZ,MasTab,SmiWei,Zor}.

The spaces $QD(S)$ and $\mathsf{Flat}(S)$ are both naturally stratified by the cone point, cone angle, and holonomy structure. For example, a structure $\sigma=(4\pi,4\pi;-1)$ describes a total of 2 cone points of cone angle $4\pi$ and the holonomy group being $\{\pm 1\}$. A stratum of $QD(S)$, denoted by $QD(S,\sigma)$, is the collection of quadratic differentials with the strata structure $\sigma$.

A convenient way to represent a quadratic differential on $S$ is via a \emph{polygons with gluing} representation. A polygons with gluing representation is a finite collection of Euclidean polygons embedded in the Euclidean plane with a pairing of all oriented edges. If $P$ is a representative of $q\in QD(S)$, then the topological surface obtained by identifying the pairs is $S$, any pair of edges in the pairing has the same length and slope, and the number of points with cone angle $\pi$ is at most $p$. Two polygons with gluing representations are equivalent if there exists a common refinement such that the natural isometry between polygons is homotopic to the identity map on $S$.

The polygons with gluing representation of a quadratic differential makes the action of $GL(2,\mathbb{R})$ on $QD(S)$ natural. We use the action when normalizing the area of a quadratic differential. The space $\mathsf{Flat}(S)$ can also be defined to be $QD(S) / {{\cos\theta\ \ -\sin\theta}\choose{\sin\theta \ \ \ \cos\theta}}$. The action of $SL(2,\mathbb{R})$ is studied extensively in \cites{EMM,Wri}.

Denote the space of homotopy classes of (non-peripheral, non-nullhomotopic) simple closed curves on $S$ by $\mathcal{S}(S)$. We use $\mathcal{MF}(S)$ ($\mathcal{PMF}(S)$) to denote the space of (projective) measured foliations classes. A measured foliation is a singular foliation on the surface with finitely many singularities and a transverse measure defined. Abusing notation, I will treat a simple closed curve as a homotopy class and a measured foliation as a class of measured foliations that are equivalent under isotopy and Whitehead moves. Projective measured foliation classes are the result of applying a further quotient by positive constant multiples of the measure. Any measured foliation can be realized as a foliation of a quadratic differential via vertical lines, equipped with $|dx|$ as the measure.

The work of Thurston \cite{FLP} illustrates how $\mathcal{MF}(S)$ is seen as the limit of weighted simple closed curves with respect to the intersection number. The intersection function extends continuously to pairs of measured foliations.
\[
\iota : \mathcal{MF}(S) \times \mathcal{MF}(S) \to \mathbb{R}_{\geq 0}.
\]
We say that a pair of distinct measured foliations $(\mu,\nu)$ is \emph{filling} if for any simple closed curve $\alpha \in \mathcal{S}(S)$, 
\[
\iota(\alpha,\mu) + \iota(\alpha,\nu) > 0.
\]
Measured foliations (and the equivalent measured laminations) play important roles in Teichm\"{u}ller theory and low-dimensional topology, see \cites{BonOta,HubMas,Mas,Thu2} for a few applications.

The work of Gardiner-Masur \cite{GarMas} shows that the space of quadratic differentials is homeomorphic to the space of pairs of filling measured foliations $\Omega$. 
\[
\Omega = \{(\mu,\nu) \in \mathcal{MF}(S) \times \mathcal{MF}(S) \mid (\mu,\nu) \text{ is filling}\}.
\] 
The correspondence is obtained by observing the horizontal and vertical foliations of the quadratic differential. Hence a unit-area quadratic differential maps to a pair of filling measured foliations that have intersection number equal to one.
\[
\Omega_1 = \{(\mu,\nu) \in \mathcal{MF}(S) \times \mathcal{MF}(S) \mid (\mu,\nu) \text{ is filling}, \iota(\mu,\nu) = 1\}.
\] 
We will denote this correspondence by 
\[
\mathcal{F} : QD(S) \to \Omega \subset \mathcal{MF}(S) \times \mathcal{MF}(S),\ \mathcal{F}(q) =(\mathcal{F}_H(q), \mathcal{F}_V(q)).
\]
The topology on the space of measured foliations can be given by intersection number with simple closed curves. The topology of $QD(S)$ is given by the product topology obtained from $\mathcal{MF}(S)$.

Recall that the slice $QD^\mu(S)$ is defined to be the set of quadratic differentials with a prescribed horizontal foliation. In other words,
\[
QD^\mu(S) = \Omega \cap (\mu, \mathcal{MF}(S)).
\]
The Euclidean cone metrics induced by quadratic differentials in $QD^\mu(S)$ is bijectively equivalent to $QD^\mu(S)$ since rotation changes the horizontal foliation. The slices $QD^{\mu}(S)$, $\mu\in\mathcal{MF}(S)$, form a foliation of $QD(S)$ and, abusing notation, a quadratic differential in a slice is identified with the induced metric in $\mathsf{Flat}(S)$. 

For a fixed $q \in QD(S)$, any simple closed curve $\alpha \in \mathcal{S}(S)$ has a geodesic representative (with respect to the induced metric in $\mathsf{Flat}(S)$) that consists of a sequence of \emph{saddle connections}, line segments connecting cone points or line segments that form a closed curve. The geodesic representative of a simple closed curve may not be simple as self-intersections can occur at cone points or repeated saddle connections. A sequence of saddle connections is the geodesic representative of a closed curve if consecutive saddle connections satisfy the \emph{angle condition} at the cone point in between. The angle condition requires that the angle on both sides of the curve is greater or equal to $\pi$ (only one side is required at marked points).

The set of \emph{cylinder curves} is the set of simple closed curves that does not have a unique geodesic representative. If $\alpha$ is a cylinder curve, then the corresponding cylinder $C_\alpha$ is the union of the geodesic representatives of $\alpha$. The sequence of saddle connections for a cylinder curve can be chosen to be either boundary of the cylinder. The length of a simple closed curve $\alpha$ refers to the length of a geodesic representative of $\alpha$, denoted by $\ell_q(\alpha)$ as it depends on $q$.

We choose to describe/distinguish simple closed curves and measured foliations using the following proposition.

\begin{prop}\label{PropInt}
The intersection function $\iota_{\mathcal{S}(S)} : \mathcal{MF}(S) \to \mathbb{R}^{\mathcal{S}(S)}_{\geq 0}$ is injective.
\end{prop}

In other words, a simple closed curve or a measured foliation is uniquely determined by the intersection number with all simple closed curves. As a result any quadratic differential is uniquely determined by $\iota_{\mathcal{S}(S)}\circ\mathcal{F}_H$ and $\iota_{\mathcal{S}(S)}\circ\mathcal{F}_V$.


\section{Examples of flat grafting}\label{EXAMP}

We begin by giving some examples of flat grafting maps on the torus. We consider a standard square torus $T$ obtained by $\mathbb{R}^2 / \mathbb{Z}^2$. We describe a type of flat grafting map called a \emph{horizontal flat grafting} along a simple closed curve. Suppose that $\alpha$ is a simple closed curve on $T$, represented by an integral vector $(m,n)$ with $\operatorname{gcd}(m,n)=1$. In the case when $n\neq 0$, we cut along $\alpha$ and glue in a cylinder, see Figure~\ref{ex1}. Geometrically, the cylinder is represented by a parallelogram of sides $(1,0)$ and $(m,n)$. The top side and the bottom side (corresponding to $(1,0)$) are identified by translation. The left side and the right side are identified with the two boundary components we obtain when cutting along $\alpha$. The process above is called the horizontally flat grafting of $q$ along $\alpha$ for time $1$.

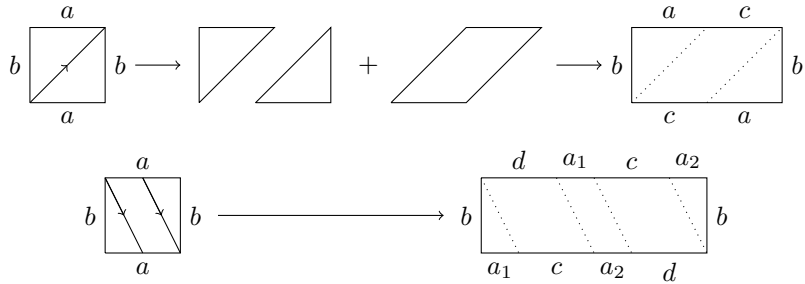
\begin{figure}[ht]
\begin{tikzpicture}
\draw (0,0) -- (0.5,0) node[below]{$a$} -- (1,0) -- (1,0.5) node[right]{$b$} -- (1,1) -- (0.5,1) node[above]{$a$} -- (0,1) -- (0,0.5) node[left]{$b$} -- (0,0);
\draw [->] (0,0) -- (0.5,0.5);
\draw (0.5,0.5) -- (1,1);

\draw [->] (1.4,0.5) -- (2,0.5);

\draw (2.25,0) -- (2.75,0.5) -- (3.25,1) -- (2.75,1) -- (2.25,1) -- (2.25,0.5) -- (2.25,0);
\draw (3,0) -- (3.5,0) -- (4,0) -- (4,0.5) -- (4,1) -- (3.5,0.5) -- (3,0);

\draw (4.8,0) -- (5.3,0) -- (5.8,0) -- (6.3,0.5) -- (6.8,1) -- (6.3,1) -- (5.8,1) -- (5.3,0.5) -- (4.8,0);
\draw (4.5,0.5) node{+};

\draw [->] (7,0.5) -- (7.6,0.5);

\draw (8,0) -- (8.5,0) node[below]{$c$} -- (9.5,0) node[below]{$a$} -- (10,0) -- (10,0.5) node[right]{$b$} -- (10,1) -- (9.5,1) node[above]{$c$} -- (8.5,1) node[above]{$a$} -- (8,1) -- (8,0.5) node[left]{$b$} -- (8,0);
\draw [dotted] (8,0) -- (9,1);
\draw [dotted] (9,0) -- (10,1);

\draw (1,-2) -- (1.5,-2) node[below]{$a$} -- (2,-2) -- (2,-1.5) node[right]{$b$} -- (2,-1) -- (1.5,-1) node[above]{$a$} -- (1,-1) -- (1,-1.5) node[left]{$b$} -- (1,-2);
\draw [->] (1,-1) -- (1.25,-1.5);
\draw (1,-1) -- (1.5,-2);
\draw [->] (1.5,-1) -- (1.75,-1.5);
\draw (1.5,-1) -- (2,-2);

\draw [->] (2.5,-1.5) -- (5.5,-1.5);

\draw (6,-1) -- node[left]{$b$} (6,-2) -- node[below]{$a_1$} (6.5,-2) -- node[below]{$c$} (7.5,-2) -- node[below]{$a_2$} (8,-2) -- node[below]{$d$} (9,-2) -- node[right]{$b$} (9,-1) -- node[above]{$a_2$} (8.5,-1) -- node[above]{$c$} (7.5,-1) -- node[above]{$a_1$} (7,-1) -- node[above]{$d$} (6,-1);

\draw [dotted] (6,-1) -- (6.5,-2);
\draw [dotted] (7,-1) -- (7.5,-2);
\draw [dotted] (7.5,-1) -- (8,-2);
\draw [dotted] (8.5,-1) -- (9,-2);

\end{tikzpicture}
\caption{Examples of flat grafting with $(m,n) = (1,1)$ and $(1,-2)$.}
\label{ex1}
\end{figure}

We extend the horizontal flat grafting to the case when $n=0$. Observe that the degenerate case can be defined in two ways: shearing the rectangle to the left or to the right. The choice implies that horizontal flat grafting maps should not extend continuously to grafting along measured foliations. Since we are working to provide analogs of earthquake and grafting deformations, we choose the shearing in degenerate cases so that we actually see a left-earthquake, see Figure~\ref{ex2}. 

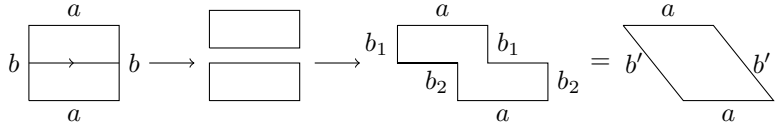
\begin{figure}[ht]
\begin{tikzpicture}
\draw (0,0) -- (0.6,0) node[below]{$a$} -- (1.2,0) -- (1.2,0.5) node[right]{$b$} -- (1.2,1) -- (0.6,1) node[above]{$a$} -- (0,1) -- (0,0.5) node[left]{$b$} -- (0,0);
\draw [->] (0,0.5) -- (0.6,0.5);
\draw (0.6,0.5) -- (1.2,0.5);

\draw [->] (1.6,0.5) -- (2.2,0.5);

\draw (2.4,0) -- (3.6,0) -- (3.6,0.5) -- (2.4,0.5) -- (2.4,0);
\draw (2.4,0.7) -- (3.6,0.7) -- (3.6,1.2) -- (2.4,1.2) -- (2.4,0.7);

\draw [->] (3.8,0.5) -- (4.4,0.5);

\draw (4.9,1) -- (5.5,1) node[above]{$a$} -- (6.1,1) -- (6.1,0.75) node[right]{$b_1$} -- (6.1,0.5) -- (6.9,0.5) -- (6.9,0.25) node[right]{$b_2$} -- (6.9,0) -- (6.3,0) node[below]{$a$} -- (5.7,0) -- (5.7,0.25) node[left]{$b_2$} -- (5.7,0.5) -- (4.9,0.5) -- (5.7,0.5) -- (4.9,0.5) -- (4.9,0.75) node[left]{$b_1$} -- (4.9,1) ;

\draw (7.6,0.5) node{=};

\draw (7.9,1) -- (8.5,1) node[above]{$a$} -- (9.1,1) -- (9.5,0.5) node[right]{$b'$} -- (9.9,0) -- (9.3,0) node[below]{$a$} -- (8.7,0) -- (8.3,0.5) node[left]{$b'$} -- (7.9,1) ;

\end{tikzpicture}
\caption{Example of flat grafting along a horizontal curve.}
\label{ex2}
\end{figure}

For a surface $S$ with $g\geq 2$, we first see that the same construction on the torus can be applied to cylinder curves. Let $q$ be a quadratic differential with a polygons with gluing representative $P$. If $\alpha$ is a cylinder curve, then $C_\alpha$ is a parallelogram in $P$ with $\alpha$ homotopic to the boundary of $C_\alpha$. We cut along a geodesic representative of $\alpha$ and glue in a cylinder. Notice that the result is the same regardless of which geodesic representative of $\alpha$ we cut. See Figure~\ref{ex3} for an example.

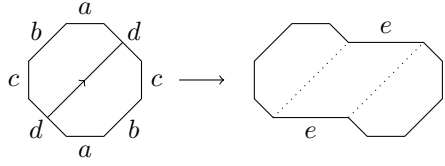
\begin{figure}[ht]
\begin{tikzpicture}

\draw (0.5,0) -- (0.75,0) node[below]{$a$} -- (1,0) -- (1.5,0.5) -- (1.5,0.75) node[right]{$c$} -- (1.5,1) -- (1,1.5) -- (0.75,1.5) node[above]{$a$} -- (0.5,1.5) -- (0,1) -- (0,0.75) node[left]{$c$} -- (0,0.5) -- (0.5,0);

\draw (0.1,0.1) node{$d$};
\draw (1.4,1.4) node{$d$};
\draw (1.4,0.1) node{$b$};
\draw (0.1,1.4) node{$b$};

\draw [->] (0.25,0.25) -- (0.75,0.75);
\draw (0.75,0.75) -- (1.25,1.25);

\draw [->] (2,0.75) -- (2.6,0.75);

\draw (4.5,0) -- (4.75,0) -- (5,0) -- (5.25,0.25) -- (5.5,0.5) -- (5.5,0.75) -- (5.5,1) -- (5.25,1.25) -- (4.75,1.25) node[above]{$e$} -- (4.25,1.25) -- (4,1.5) -- (3.75,1.5) -- (3.5,1.5) -- (3.25,1.25) -- (3,1) -- (3,0.75) -- (3,0.5) -- (3.25,0.25) -- (3.75,0.25) node[below]{$e$} -- (4.25,0.25) -- (4.5,0);

\draw [dotted] (3.25,0.25) -- (4.25,1.25);
\draw [dotted] (4.25,0.25) -- (5.25,1.25);

\end{tikzpicture}
\caption{Example of flat grafting along a cylinder curve.}
\label{ex3}
\end{figure}

We provide three more examples corresponding to cases when $\alpha$ is not a cylinder curve. The number of parallelograms we glue in after cutting along $\alpha$ is equal to the number of saddle connections in the geodesic representative of $\alpha$. The parallelograms are glued by isometries that respect the orientation. See Figure~\ref{ex456} for the three examples. The examples take a quadratic differential in the stratum $QD(S,(6\pi;1))$ to quadratic differentials in the strata $QD(S,(4\pi,4\pi;1))$, $QD(S,(4\pi,3\pi,3\pi;1))$, and $QD(S,(4\pi,3\pi,3\pi;-1))$ respectively.

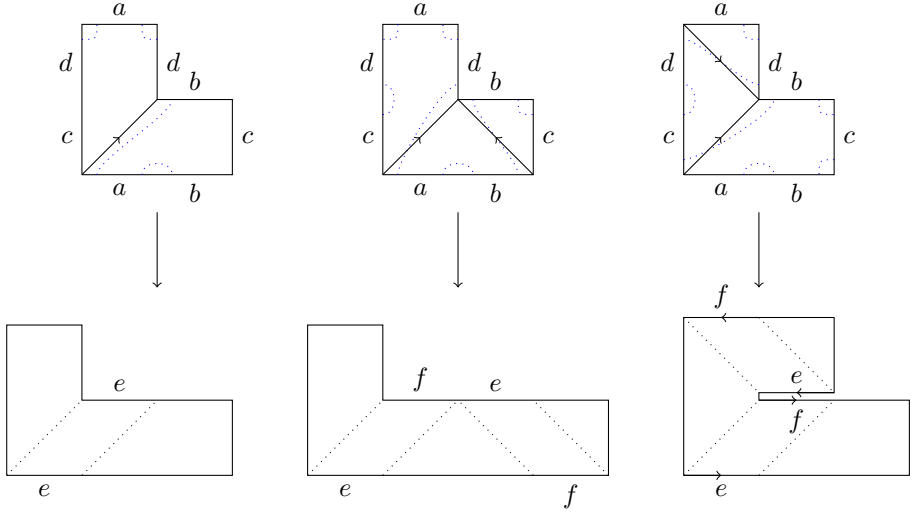
\begin{figure}[ht]
\begin{tikzpicture}

\draw (0,10) -- node[below]{$a$} (1,10) -- node[below]{$b$} (2,10) -- node[right]{$c$} (2,11) -- node[above]{$b$} (1,11) -- node[right]{$d$} (1,12) -- node[above]{$a$} (0,12) -- node[left]{$d$} (0,11) -- node[left]{$c$} (0,10);

\draw[dotted,blue] (0,11.8) .. controls (0.2,11.8) and (0.2,11.8) .. (0.2,12);
\draw[dotted,blue] (1,11.8) .. controls (0.8,11.8) and (0.8,11.8) .. (0.8,12);
\draw[dotted,blue] (0.8,10) .. controls (0.8,10.2) and (1.2,10.2) .. (1.2,10);
\draw[dotted,blue] (0.2,10) .. controls (0.2,10.2) and (1.2,10.8) .. (1.2,11);

\draw [->] (0,10) -- (0.5,10.5);
\draw (0.5,10.5) -- (1,11);

\draw [->] (1,9.5) -- (1,8.5);

\draw (-1,6) -- node[below]{$e$} (0,6) -- (1,6) -- (2,6) -- (2,7) -- (1,7) -- node[above]{$e$} (0,7) -- (0,8) -- (-1,8) -- (-1,7) -- (-1,6);
\draw [dotted] (-1,6) -- (0,7);
\draw [dotted] (0,6) -- (1,7);

\draw (4,10) -- node[below]{$a$} (5,10) -- node[below]{$b$} (6,10) -- node[right]{$c$} (6,11) -- node[above]{$b$} (5,11) -- node[right]{$d$} (5,12) -- node[above]{$a$} (4,12) -- node[left]{$d$} (4,11) -- node[left]{$c$} (4,10);

\draw[dotted,blue] (4.2,10) .. controls (4.2,10.2) and (4.8,11.2) .. (5,11.2);
\draw[dotted,blue] (4.2,12) .. controls (4.2,11.8) and (4.2,11.8) .. (4,11.8);
\draw[dotted,blue] (5,11.8) .. controls (4.8,11.8) and (4.8,11.8) .. (4.8,12);
\draw[dotted,blue] (4.8,10) .. controls (4.8,10.2) and (5.2,10.2) .. (5.2,10);
\draw[dotted,blue] (5.2,11) .. controls (5.2,10.8) and (5.8,10.2) .. (5.8,10);
\draw[dotted,blue] (5.8,11) .. controls (5.8,10.8) and (5.8,10.8) .. (6,10.8);
\draw[dotted,blue] (4,10.8) .. controls (4.2,10.8) and (4.2,11.2) .. (4,11.2);

\draw [->] (4,10) -- (4.5,10.5);
\draw (4.5,10.5) -- (5,11);
\draw [->] (6,10) -- (5.5,10.5);
\draw (5.5,10.5) -- (5,11);

\draw [->] (5,9.5) -- (5,8.5);

\draw (3,6) -- node[below]{$e$} (4,6) -- (5,6) -- (6,6) -- node[below]{$f$} (7,6) -- (7,7) -- (6,7) -- node[above]{$e$} (5,7) -- node[above]{$f$} (4,7) -- (4,8) -- (3,8) -- (3,7) -- (3,6);
\draw [dotted] (3,6) -- (4,7);
\draw [dotted] (4,6) -- (5,7);
\draw [dotted] (6,6) -- (5,7);
\draw [dotted] (7,6) -- (6,7);

\draw (8,10) -- node[below]{$a$} (9,10) -- node[below]{$b$} (10,10) -- node[right]{$c$} (10,11) -- node[above]{$b$} (9,11) -- node[right]{$d$} (9,12) -- node[above]{$a$} (8,12) -- node[left]{$d$} (8,11) -- node[left]{$c$} (8,10);

\draw[dotted,blue] (8,10.2) .. controls (8.2,10.2) and (9.2,10.8) .. (9.2,11);
\draw[dotted,blue] (10,10.2) .. controls (9.8,10.2) and (9.8,10.2) .. (9.8,10);
\draw[dotted,blue] (9.8,11) .. controls (9.8,10.8) and (9.8,10.8) .. (10,10.8);
\draw[dotted,blue] (8,10.8) .. controls (8.2,10.8) and (8.2,11.2) .. (8,11.2);
\draw[dotted,blue] (9,11.2) .. controls (8.8,11.2) and (8.2,11.8) .. (8,11.8);
\draw[dotted,blue] (9,11.8) .. controls (8.8,11.8) and (8.8,11.8) .. (8.8,12);
\draw[dotted,blue] (8.8,10) .. controls (8.8,10.2) and (9.2,10.2) .. (9.2,10);

\draw [->] (8,10) -- (8.5,10.5);
\draw (8.5,10.5) -- (9,11);
\draw [->] (8,12) -- (8.5,11.5);
\draw (8.5,11.5) -- (9,11);

\draw [->] (9,9.5) -- (9,8.5);

\draw (8,6) -- node[below]{$e$} (9,6) -- (10,6) -- (11,6) -- (11,7) -- (10,7) -- node[below]{$f$} (9,7) -- (9,7.1) -- node[above]{$e$} (10,7.1) -- (10,8.1) -- (9,8.1) -- node[above]{$f$} (8,8.1) -- (8,7) -- (8,6);
\draw [dotted] (8,8.1) -- (9,7.1);
\draw [dotted] (9,8.1) -- (10,7.1);
\draw [dotted] (8,6) -- (9,7);
\draw [dotted] (9,6) -- (10,7);
\draw [->] (8,6) -- (8.5,6);
\draw [->] (10,7.1) -- (9.5,7.1);
\draw [->] (9,7) -- (9.5,7);
\draw [->] (9,8.1) -- (8.5,8.1);

\end{tikzpicture}
\caption{Examples of flat grafting along non-cylinder curves.}
\label{ex456}
\end{figure}

A horizontal flat grafting specifies the sides of the parallelograms added during the grafting process. A general flat grafting will be formally introduced in Section~\ref{CONST}. Figure~\ref{ex789} shows some examples of non-horizontal flat grafting. The resulting stratum may differ even when grafting along the same simple closed curve. Observe that we can rotate the examples to view them as examples of horizontal flat grafting.

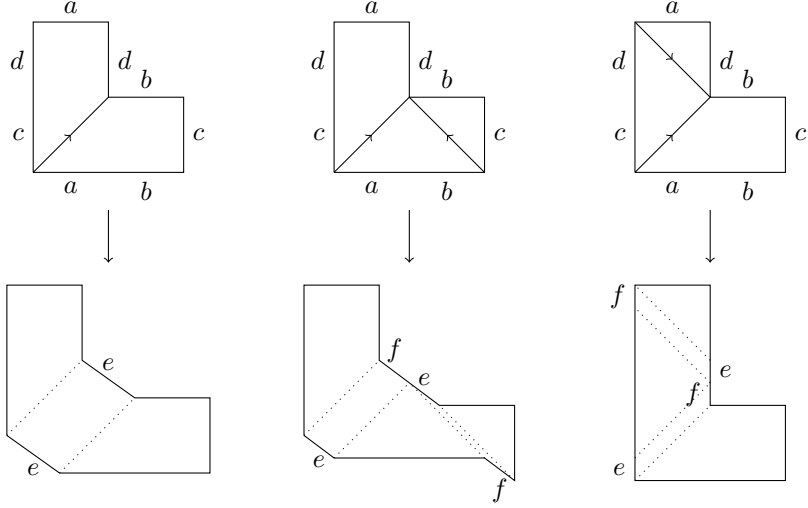
\begin{figure}[ht]
\begin{tikzpicture}

\draw (0,10) -- node[below]{$a$} (1,10) -- node[below]{$b$} (2,10) -- node[right]{$c$} (2,11) -- node[above]{$b$} (1,11) -- node[right]{$d$} (1,12) -- node[above]{$a$} (0,12) -- node[left]{$d$} (0,11) -- node[left]{$c$} (0,10);

\draw [->] (0,10) -- (0.5,10.5);
\draw (0.5,10.5) -- (1,11);

\draw [->] (1,9.5) -- (1,8.8);

\draw (-0.35,6.5) -- node[below]{$e$} (0.35,6) -- (1.35,6) -- (2.35,6) -- (2.35,7) -- (1.35,7) -- node[above]{$e$} (0.65,7.5) -- (0.65,8.5) -- (-0.35,8.5) -- (-0.35,7.5) -- (-0.35,6.5);
\draw [dotted] (-0.35,6.5) -- (0.65,7.5);
\draw [dotted] (0.35,6) -- (1.35,7);

\draw (4,10) -- node[below]{$a$} (5,10) -- node[below]{$b$} (6,10) -- node[right]{$c$} (6,11) -- node[above]{$b$} (5,11) -- node[right]{$d$} (5,12) -- node[above]{$a$} (4,12) -- node[left]{$d$} (4,11) -- node[left]{$c$} (4,10);

\draw [->] (4,10) -- (4.5,10.5);
\draw (4.5,10.5) -- (5,11);
\draw [->] (6,10) -- (5.5,10.5);
\draw (5.5,10.5) -- (5,11);

\draw [->] (5,9.5) -- (5,8.8);

\draw (3.6,6.5) -- node[below]{$e$} (4,6.2) -- (5,6.2) -- (6,6.2) -- node[below]{$f$} (6.4,5.9) -- (6.4,6.9) -- (5.4,6.9) -- node[above]{$e$} (5,7.2) -- node[above]{$f$} (4.6,7.5) -- (4.6,8.5) -- (3.6,8.5) -- (3.6,7.5) -- (3.6,6.5);
\draw [dotted] (3.6,6.5) -- (4.6,7.5);
\draw [dotted] (4,6.2) -- (5,7.2);
\draw [dotted] (6,6.2) -- (5,7.2);
\draw [dotted] (6.4,5.9) -- (5.4,6.9);

\draw (8,10) -- node[below]{$a$} (9,10) -- node[below]{$b$} (10,10) -- node[right]{$c$} (10,11) -- node[above]{$b$} (9,11) -- node[right]{$d$} (9,12) -- node[above]{$a$} (8,12) -- node[left]{$d$} (8,11) -- node[left]{$c$} (8,10);

\draw [->] (8,10) -- (8.5,10.5);
\draw (8.5,10.5) -- (9,11);
\draw [->] (8,12) -- (8.5,11.5);
\draw (8.5,11.5) -- (9,11);

\draw [->] (9,9.5) -- (9,8.8);

\draw (8,6.2) -- node[left]{$e$} (8,5.9) -- (9,5.9) -- (10,5.9) -- (10,6.9) -- (9,6.9) -- node[left]{$f$} (9,7.2) -- node[right]{$e$} (9,7.5) -- (9,8.5) -- (8,8.5) -- node[left]{$f$} (8,8.2) -- (8,7.2) -- (8,6.2);
\draw [dotted] (8,6.2) -- (9,7.2);
\draw [dotted] (8,5.9) -- (9,6.9);
\draw [dotted] (9,7.2) -- (8,8.2);
\draw [dotted] (9,7.5) -- (8,8.5);

\end{tikzpicture}
\caption{Examples of non-horizontal flat grafting.}
\label{ex789}
\end{figure}


\section{Construction and normalization}\label{CONST}

We define the horizontal flat grafting function 
\[
HG : QD(S) \times \mathcal{S}(S) \to QD(S).
\]
Let $q \in QD(S)$ be a fixed quadratic differential and $\alpha \in \mathcal{S}(S)$ be a simple closed curve on $S$. We define $HG(q,\alpha)$ via polygons with gluing representative. Pick a geodesic representative of $\alpha$. Let $s_1, s_2, \ldots, s_m$ be a sequence of saddle connections for the geodesic representative. Give $\alpha$ an orientation and let $s_j = (x_j,y_j)$ be vectors in $\mathbb{R}^2$.

Topologically, we cut the surface along the simple closed curve to get two boundary components, and then we glue in an annulus. Geometrically, we cut the polygons along the saddle connections $s_1$ to $s_m$ to obtain $2m$ boundary saddle connections. Add a parallelogram described by the vectors $(1,0)$ and $s_j$ for each $1 \leq j \leq m$. We glue the left-sides and right-sides of the parallelograms to corresponding boundary edges. We glue the top-sides and bottom-sides to each other according to the orientation of $\alpha$. If $y_j$ and $y_{j+1}$ has the same sign, then the gluing will be a translation. If $y_j$ and $y_{j+1}$ has opposite signs, then the gluing will be a semi-translation, see the third example of Figure~\ref{ex456}. The collection of parallelograms will be called the \emph{grafting cylinder} of the flat grafting along $\alpha$.

If the geodesic representative of $\alpha$ with respect to $q$ contains a horizontal saddle connection, we define the horizontal grafting by a degenerate parallelogram similar to Figure~\ref{ex2}. The horizontal grafting is defined by the limit of rotating $q$ counterclockwise by a small angle, applying the horizontal flat grafting, then rotate back. A geodesic representative of $\alpha$ could contain multiple copies of the same saddle connection. The saddle connections are ordered by relative positioning since $\alpha$ is a simple closed curve. The parallelograms are glued based on the ordering in that case.

\begin{prop}
The construction of $HG : QD(S) \times \mathcal{S}(S) \to QD(S)$ is well-defined.
\end{prop}

\begin{proof}
Consider $HG(q,\alpha)$ with $q \in QD(S)$ and $\alpha\in \mathcal{S}(S)$. We check if the choices made during the construction affect our result. If there are more than one geodesic representative for $\alpha$, then $\alpha$ is a cylinder curve with respect to $q$. In this case $\alpha$ remains a cylinder curve with respect to $HG(q,\alpha)$. We see that the process applied to any geodesic representative will result in the same cylinder up to cutting and gluing.

Therefore the only other choice we made was the polygons with gluing representative of $q$. Suppose we have two different representatives of $q$. There exists a common refinement of the two representatives. We take all edges in both set of polygons and cut into smaller pieces of polygons. The geodesic representative in the refinement must correspond and hence the result of the construction must be the same.
\end{proof}

We add a parameter $t$ to describe the horizontal width of parallelograms being glued in. Hence the map $HG_t$ is also well-defined for all $t\geq 0$. We use $VG_t : QD(S) \times \mathcal{S}(S) \to QD(S)$ to denote the vertical grafting defined similarly except with vertical pieces. The most general form of flat grafting will be denoted $F_u$, where $u \in \mathbb{R}^2$ is the vector that both describes the direction and the width of the grafting. Hence $HG_t = F_{(t,0)}$ and $VG_t = F_{(0,t)}$. A general grafting is equivalent to conjugating a horizontal grafting by a rotation action, and therefore automatically well-defined. An inverse horizontal flat grafting map corresponding to negative values of $t$ can be constructed but the inverse map generically leads to pinching and degeneration of the quadratic differential.

Restricting the flat grafting map to unit-area quadratic differentials $QD_1(S)$, we define a normalized horizontal flat grafting map
\[
\overline{HG}_t : QD_1(S) \times \mathcal{S}(S) \to QD_1(S).
\]
Instead of scaling by the area of the resulting quadratic differential, we only scale the horizontal width, that is, 
\[
\overline{HG}_t (q,\alpha)= {{1/A\ \ 0}\choose{\ 0 \ \ \ \ 1}}HG_t(q,\alpha), A=\text{area of }HG_t(q,\alpha). 
\]
See Figure~\ref{normhf} for examples of normalized horizontal flat grafting. Similarly $\overline{VG}$ and $\overline{F}$ denote normalized vertical flat grafting and general flat grafting, respectively. 

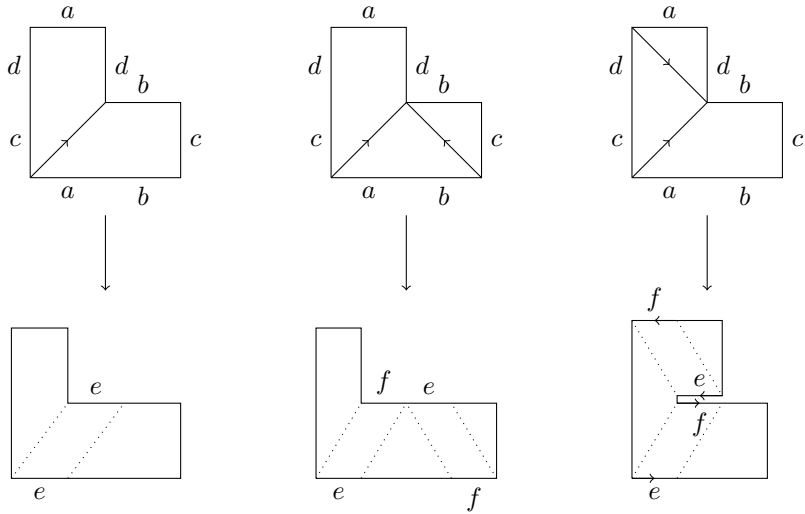
\begin{figure}[ht]
\begin{tikzpicture}

\draw (0,10) -- node[below]{$a$} (1,10) -- node[below]{$b$} (2,10) -- node[right]{$c$} (2,11) -- node[above]{$b$} (1,11) -- node[right]{$d$} (1,12) -- node[above]{$a$} (0,12) -- node[left]{$d$} (0,11) -- node[left]{$c$} (0,10);

\draw [->] (0,10) -- (0.5,10.5);
\draw (0.5,10.5) -- (1,11);

\draw [->] (1,9.5) -- (1,8.5);

\draw (-0.25,6) -- node[below]{$e$} (0.5,6) -- (1.25,6) -- (2,6) -- (2,7) -- (1.25,7) -- node[above]{$e$} (0.5,7) -- (0.5,8) -- (-0.25,8) -- (-0.25,7) -- (-0.25,6);
\draw [dotted] (-0.25,6) -- (0.5,7);
\draw [dotted] (0.5,6) -- (1.25,7);

\draw (4,10) -- node[below]{$a$} (5,10) -- node[below]{$b$} (6,10) -- node[right]{$c$} (6,11) -- node[above]{$b$} (5,11) -- node[right]{$d$} (5,12) -- node[above]{$a$} (4,12) -- node[left]{$d$} (4,11) -- node[left]{$c$} (4,10);

\draw [->] (4,10) -- (4.5,10.5);
\draw (4.5,10.5) -- (5,11);
\draw [->] (6,10) -- (5.5,10.5);
\draw (5.5,10.5) -- (5,11);

\draw [->] (5,9.5) -- (5,8.5);

\draw (3.8,6) -- node[below]{$e$} (4.4,6) -- (5,6) -- (5.6,6) -- node[below]{$f$} (6.2,6) -- (6.2,7) -- (5.6,7) -- node[above]{$e$} (5,7) -- node[above]{$f$} (4.4,7) -- (4.4,8) -- (3.8,8) -- (3.8,7) -- (3.8,6);
\draw [dotted] (3.8,6) -- (4.4,7);
\draw [dotted] (4.4,6) -- (5,7);
\draw [dotted] (5.6,6) -- (5,7);
\draw [dotted] (6.2,6) -- (5.6,7);

\draw (8,10) -- node[below]{$a$} (9,10) -- node[below]{$b$} (10,10) -- node[right]{$c$} (10,11) -- node[above]{$b$} (9,11) -- node[right]{$d$} (9,12) -- node[above]{$a$} (8,12) -- node[left]{$d$} (8,11) -- node[left]{$c$} (8,10);

\draw [->] (8,10) -- (8.5,10.5);
\draw (8.5,10.5) -- (9,11);
\draw [->] (8,12) -- (8.5,11.5);
\draw (8.5,11.5) -- (9,11);

\draw [->] (9,9.5) -- (9,8.5);

\draw (8,6) -- node[below]{$e$} (8.6,6) -- (9.2,6) -- (9.8,6) -- (9.8,7) -- (9.2,7) -- node[below]{$f$} (8.6,7) -- (8.6,7.1) -- node[above]{$e$} (9.2,7.1) -- (9.2,8.1) -- (8.6,8.1) -- node[above]{$f$} (8,8.1) -- (8,7) -- (8,6);
\draw [dotted] (8,8.1) -- (8.6,7.1);
\draw [dotted] (8.6,8.1) -- (9.2,7.1);
\draw [dotted] (8,6) -- (8.6,7);
\draw [dotted] (8.6,6) -- (9.2,7);
\draw [->] (8,6) -- (8.3,6);
\draw [->] (9.2,7.1) -- (8.9,7.1);
\draw [->] (8.6,7) -- (8.9,7);
\draw [->] (8.6,8.1) -- (8.3,8.1);

\end{tikzpicture}
\caption{Normalized horizontal flat grafting.}
\label{normhf}
\end{figure}

\noindent \textbf{Remark.} The parallelogram construction in \cites{EMZ,MZ,Boi} is a surgery that can be applied to transport holes generated by a local surgery on a flat surface. The method was developed to study homologous saddle connections and the Siegel-Veech constants. The most general construction can be found in \cite[Section~3]{Boi}. A flat grafting map is a global approach to the parallelogram construction that replaces their local criteria with the structure of simple closed curves. We illustrate this in Section~\ref{STRAT} where we discuss the topic of how grafting changes the strata structure.


\section{Properties of flat grafting}\label{PROPE}

We focus on properties of horizontal flat grafting in this section since general flat grafting can be obtained by conjugating by rotation ${{\cos\theta\ \ -\sin\theta}\choose{\sin\theta \ \ \ \cos\theta}}$. Theorem~\ref{mainprop} is a corollary of Theorem~\ref{mainpropH} and Theorem~\ref{continuityH}.

\begin{theorem}\label{mainpropH}
Let $\overline{HG}_t : QD_1(S) \times \mathcal{S}(S) \to QD_1(S)$ be the normalized horizontal flat grafting map with $t>0$.
\begin{itemize}
\item[(A)] The horizontal foliation $\mathcal{F}_H(q) \in \mathcal{MF}(S)$ is preserved under the map $\overline{HG}_t$.
\item[(B)] $\overline{HG}_t(q,\alpha) = q$ if and only if $\alpha$ is equivalent to the projective class of the vertical foliation $\mathcal{F}_V(q)$ in $\mathcal{PMF}(S)$.
\item[(C)] $\overline{HG}_t(q, \ \cdot\ )$ is injective for all $q\in QD_1(S)$.
\end{itemize}
\end{theorem}

\begin{proof}
We prove property (A) by considering the intersection function in Proposition~\ref{PropInt}. The horizontal foliation is preserved if and only if the intersection number with every simple closed curve is preserved under the map $\overline{HG}_t$. Note that the intersection number is preserved under normalization, hence we can consider the map $HG_t$. 

Fix $(q,\alpha) \in QD_1(S) \times \mathcal{S}(S)$, $t>0$, and $\gamma \in \mathcal{S}(S)$. The intersection number between $\gamma$ and the horizontal foliation $\mathcal{F}_H(q)$ is equal to $\sum |y_j|$, where $y_j$ is the $y$-value of the $j$-th saddle connection of a geodesic representative of $\gamma$ with respect to $q$. 

A geodesic representative of $\gamma$ with respect to $HG_t(q,\alpha)$ could be composed of different saddle connections when compared with a geodesic representative of $\gamma$ with respect to $q$. We consider decomposing $\gamma$ into arcs $\{A_j\}$ where each arc corresponds to a saddle connection in $q$. If the horizontal foliation $\mathcal{F}_H(HG_t(q,\alpha))$ induces a non-zero measure on the arc $A_j$, then we apply a homotopy to the arc respecting the horizontal foliation such that the interior of the arc is disjoint from the grafting cylinder. Since the measure is preserved under such a homotopy and disjoint arcs have measures induced by $\mathcal{F}_H(q)$, we conclude that the intersection number with $\alpha$, which is the sum of the total measure of the arcs, is preserved.

For property (B) we prove the easy direction first. If $\alpha$ is equivalent to the vertical foliation in $\mathcal{PMF}(S)$, then $\alpha$ is a vertical cylinder curve with $C_\alpha$ being the whole surface. Grafting along $\alpha$ increases the height of the cylinder and normalizing cancels out the increase. Therefore $\overline{HG}_t(q,\alpha) = q$.

Suppose that $\alpha$ is not equivalent to the vertical foliation in $\mathcal{PMF}(S)$. We can show that $\overline{HG}_t(q,\alpha) \neq q$ by finding a simple closed curve $\gamma$ whose length with respect to $q$ and $\overline{HG}_t(q,\alpha)$ are different. Recall that the grafting cylinder is the topological cylinder that is glued in during flat grafting. A geodesic representative of $\alpha$ with respect to $HG_t(q,\alpha)$ has to lie inside the grafting cylinder because of the angle condition, implying that 
\[
\ell_{HG_t(q,\alpha)}(\alpha) \leq \ell_q(\alpha).
\]
We observe that if a geodesic representative of $\alpha$ contains a \emph{slanted} saddle connection (neither horizontal nor vertical), then the length $\ell_{\overline{HG}_t(q,\alpha)}(\alpha)$ for $t>0$ is shorter than $\ell_q(\alpha)$. The inequality 
\[
\ell_{\overline{HG}_t(q,\alpha)}(\alpha) < \ell_{HG_t(q,\alpha)}(\alpha)
\]
is immediate due to the normalization strictly shortening a slanted saddle connection. 

This leaves us with the case of $\alpha$ having geodesic representatives consisting of only horizontal or vertical saddle connections. A similar argument works if a geodesic representative of $\alpha$ has both a horizontal and a vertical saddle connection.

If $\alpha$ is horizontal, we consider a cylinder curve $\gamma$ that intersects $\alpha$ and has negative slope. The existence of $\gamma$ follows from the denseness of cylinder curves, see \cite{Masur2}. The length of $\gamma$ with respect to $HG_t(q,\alpha)$ is strictly greater than $\ell_q(\gamma)$ due to the left twist. The area is preserved and hence 
\[
\ell_q(\gamma) < \ell_{HG_t(q,\alpha)}(\gamma) = \ell_{\overline{HG}_t(q,\alpha)}(\gamma).
\]

If $\alpha$ is vertical, then $\alpha$ is a cylinder curve in $\overline{HG}_t(q,\alpha)$ regardless of whether $\alpha$ is a cylinder curve in $q$ or not. We compare the dimensions of $C_\alpha$ to show that $\overline{HG}_t(q,\alpha) \neq q$. Let the height of $C_\alpha$ in $q$ be $K$, where $K=0$ if $\alpha$ is not a cylinder curve in $q$. Height of $C_\alpha$ in $HG_t(q,\alpha)$ is $t+K$. Height of $C_\alpha$ in $\overline{HG}_t(q,\alpha)$ is $(t+K)/A$, where 
\[
A = 1 + t\cdot\iota(\alpha,\mathcal{F}_H(q)) = 1 + t\cdot\ell_q(\alpha).
\] 
Hence the height is different as long as $K$ is not $1/\ell_q(\alpha)$ as $t>0$. Equality occurs when $\alpha$ is a vertical cylinder curve in $q$ with $C_\alpha$ being the whole surface, that is, $\alpha$ is equivalent to the vertical foliation in $\mathcal{PMF}(S)$.

Property (C) shows that horizontal grafting along distinct simple closed curves result in distinct quadratic differentials. Let $\alpha_1 \neq \alpha_2 \in\mathcal{S}(S)$ and assume that $\overline{HG}_{t}(q,\alpha_1) = \overline{HG}_{t}(q,\alpha_2)$ for some $t>0$.

Suppose that $\alpha_1$ is a cylinder curve with respect to $q$. Then a quick comparison of the dimensions of $C_{\alpha_1}$ with respect to $\overline{HG}_{t}(q,\alpha_1)$ and $\overline{HG}_{t}(q,\alpha_2)$ will prove the desired contradiction $\alpha_1=\alpha_2$. This allows us to assume that $\alpha_1$ and $\alpha_2$ both have a unique geodesic representative. 

Let $L$ be a horizontal line segment in $q$ away from singularities where the endpoints of $L$ do not lie on geodesic representatives of $\alpha_1$ and $\alpha_2$. If $\overline{HG}_{t}(q,\alpha_1)=\overline{HG}_{t}(q,\alpha_2)$, then the measures induced by the vertical foliations need to be equal for all choices of $L$. In other words, we define the notion of the length of $L$ and use it to show that $\alpha_1$ and $\alpha_2$ have geodesic representatives consisting of the same set of saddle connections, counting multiplicity.

Along a fixed $\alpha$, the measure is scaled by the normalization if $L$ is disjoint from the grafting cylinder. We use $\iota(L,\alpha)$ to denote the intersection number between $L$ and a geodesic representative of $\alpha$. The intersection number is well-defined and finite by our choice of $L$. We use $\ell_q(L)$ to denote the length of $L$, which is only well-defined if $q$ is fixed. The length of $L$ with respect to $\overline{HG}_t(q,\alpha)$ is defined by adding segments of $L$ in the grafting cylinder. Therefore 
\[
\ell_{\overline{HG}_t(q,\alpha)}(L) = \iota(L,\mathcal{F}_V(\overline{HG}_t(q,\alpha)) = \frac{\ell_q(L) + t\cdot\iota(L,\alpha)}{1+t\cdot\iota(\alpha,\mathcal{F}_H(q))}.
\]
Returning to the comparison of $\overline{HG}_{t}(q,\alpha_1)$ and $\overline{HG}_{t}(q,\alpha_2)$, we conclude that the set of saddle connections in the geodesic representatives of $\alpha_1$ and $\alpha_2$ can differ only in horizontal ones.

Suppose that $s$ is a horizontal saddle connection that has a different multiplicity with respect to $\alpha_1$ and $\alpha_2$. We can use a slanted line segment $L$ that intersects $s$ in the above argument. The induced measure on $L$ is obtained by projection. Again by considering $\iota(L,\mathcal{F}_V(\overline{HG}_t(q,\alpha))$ we conclude that $\alpha_1$ and $\alpha_2$ consist of the same set of saddle connections, counting multiplicity.

We shift our focus to the leaves of the vertical measure foliation of $q$ and how it changes under horizontal grafting. For each fixed $t$, $\overline{HG}_t(q,\alpha)$ defines a map $f$ from $\alpha$ to $\alpha$ using the grafting cylinder as in Figure~\ref{GraftCyl}. The map $f$ is the first return map in the vertical upward direction and let $f(x)=x$ if $x$ lies in the interior of a vertical saddle connection. 

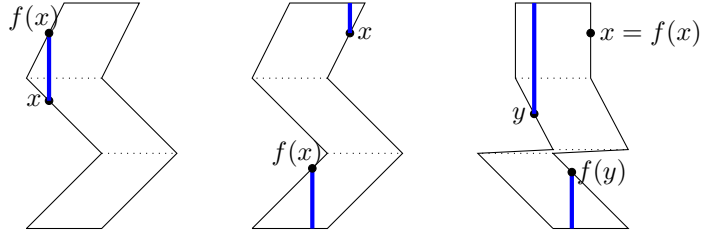
\begin{figure}[ht]
\begin{tikzpicture}

\draw (-3,0) -- (-2,0) -- (-1,1) -- (-2,2) -- (-1.5,3) -- (-2.5,3) -- (-3,2) -- (-2,1) -- (-3,0);
\draw[dotted] (-2,1) -- (-1,1);
\draw[dotted] (-2,2) -- (-3,2);
\draw (-2.9,1.7) node{$x$}; \draw[fill] (-2.7,1.7) circle (0.05);
\draw[blue,ultra thick] (-2.7,1.7) -- (-2.7,2.6);
\draw (-2.9,2.8) node{$f(x)$}; \draw[fill] (-2.7,2.6) circle (0.05);

\draw (0,0) -- (1,0) -- (2,1) -- (1,2) -- (1.5,3) -- (0.5,3) -- (0,2) -- (1,1) -- (0,0);
\draw[dotted] (1,1) -- (2,1);
\draw[dotted] (1,2) -- (0,2);
\draw (1.5,2.6) node{$x$}; \draw[fill] (1.3,2.6) circle (0.05);
\draw[blue,ultra thick] (1.3,2.6) -- (1.3,3);
\draw[blue,ultra thick] (0.8,0) -- (0.8,0.8);
\draw (0.6,1) node{$f(x)$}; \draw[fill] (0.8,0.8) circle (0.05);

\draw (4,0) -- (5,0) -- (4,1) -- (5,1.05) -- (4.5,2) -- (4.5,3) -- (3.5,3) -- (3.5,2) -- (4,1.05) -- (3,1) -- (4,0);
\draw[dotted] (4,1.05) -- (5,1.05);
\draw[dotted] (4,1) -- (3,1);
\draw[dotted] (3.5,2) -- (4.5,2);
\draw[fill] (4.5,2.6) node[right]{$x=f(x)$} circle (0.05);
\draw (3.55,1.5) node{$y$}; \draw[fill] (3.75,1.525) circle (0.05);
\draw[blue,ultra thick] (3.75,1.525) -- (3.75,3);
\draw[blue,ultra thick] (4.25,0) -- (4.25,0.75);
\draw (4.65,0.75) node{$f(y)$}; \draw[fill] (4.25,0.75) circle (0.05);

\end{tikzpicture}
\caption{Grafting cylinder defines an automorphism on $\alpha$.}
\label{GraftCyl}
\end{figure}

The map $f$ is a piecewise linear automorphism on the disjoint union of saddle connections. Outside of vertical saddle connections, the map $f$ detects the combinatorics of the order of saddle connections in the geodesic representative. In other words, $\alpha_1$ and $\alpha_2$ consist of the same set of saddle connections, counting multiplicity, and the ordering of the saddle connections only differ in the placement of the vertical saddle connections.

We generalize the first return map $f$ based on the grafting cylinder to $f^\theta$, which uses first return in the $\theta$ direction. This will solve the problem with vertical saddle connections. If $\alpha_1\neq \alpha_2$, then by choosing a suitable angle $\theta$ we see that $f^\theta$ is different corresponding to $\overline{HG}_{t}(q,\alpha_1)$ and $\overline{HG}_{t}(q,\alpha_2)$. Hence we obtain the contradiction that $\overline{HG}_{t}(q,\alpha_1) = \overline{HG}_{t}(q,\alpha_2)$ implies $\alpha_1=\alpha_2$ and therefore $\overline{HG}_t(q, \ \cdot\ )$ is injective.
\end{proof}

We made a remark on the continuity of the flat grafting map in Section~\ref{EXAMP}. Here we prove a more detailed statement regarding continuity.

\begin{theorem}\label{continuityH}
Let $\overline{HG}_t : QD_1(S) \times \mathcal{S}(S)\to QD_1(S)$ be the normalized horizontal flat grafting map. 
\begin{itemize}
\item The map $\overline{HG}_\cdot(q,\alpha) : \mathbb{R}_{\geq 0} \to QD_1(S)$ is continuous for fixed $q\in QD_1(S)$ and $\alpha\in\mathcal{S}(S)$.
\item The map $\overline{HG}_t(\ \cdot\ ,\alpha)$ for fixed $t > 0$ and  $\alpha\in\mathcal{S}(S)$ is continuous at $q \in QD_1(S)$ if and only if no geodesic representatives of $\alpha$ with respect to $q$ contain a horizontal saddle connection.
\end{itemize}
\end{theorem}

\begin{proof}
Recall the topology of $QD_1(S)$ using the topology of $\mathcal{MF}(S)$. The base of the topology can be given by $B_{q_0}(K,\varepsilon)$, the collection of $q\in QD_1(S)$ where for any simple closed curve $\gamma \in \mathcal{S}(S)$ with length $\ell_q(\gamma)<K$ and $\ell_{q_0}(\gamma)<K$, we have $| \iota(\gamma,\mathcal{F}_H(q)) - \iota(\gamma,\mathcal{F}_H(q_0))| < \varepsilon$ and $| \iota(\gamma,\mathcal{F}_V(q)) - \iota(\gamma,\mathcal{F}_V(q_0))| < \varepsilon$.

Fix $q\in QD_1(S)$, $\alpha\in\mathcal{S}(S)$ and let $q_0 = \overline{HG}_{t_0}(q,\alpha)$ for some $t_0>0$. We want to show that for any $B_{q_0}(K,\varepsilon)$ there exists $\delta>0$ such that $\overline{HG}_t(q,\alpha) \in B_{q_0}(K,\varepsilon)$ for all $t \in (t_0-\delta, t_0+\delta)$. The case when $t_0=0$ is a one-sided case of the argument below. By Theorem~\ref{mainpropH}, the horizontal foliation is preserved and it suffices to estimate $\iota(\gamma,\mathcal{F}_V(\overline{HG}_t(q,\alpha)))$ for some $\gamma\in\mathcal{S}(S)$ with $\ell_{q_0}(\gamma)<K$. If $\gamma$ is a cylinder curve in $q_0$, then we can choose $\delta$ small enough such that $\gamma$ is a cylinder curve with respect to $\overline{HG}_{t}(q,\alpha)$ for any $t\in (t_0-\delta, t_0+\delta)$. The intersection number with the vertical foliation is a continuous function of $t$ that depends on $\iota(\gamma,\alpha)$ and the scaling factor $1+t\cdot \iota(\alpha,\mathcal{F}_H(q))$. 

In the case when $\gamma$ is not a cylinder curve in $q_0$, we claim that the function $\iota(\gamma,\mathcal{F}_V(\overline{HG}_t(q,\alpha)))$ is also continuous in the sense that $\delta$ can be chosen depending on $\iota(\gamma,\alpha)$ and the scaling factor $1+t\cdot \iota(\alpha,\mathcal{F}_H(q))$. The number of saddle connections in a geodesic representative of $\gamma$ with respect to $\overline{HG}_{t}(q,\alpha)$ for $t\in (t_0-\delta, t_0+\delta)$ is finite. The $x$-component of each saddle connection changes with respect to $t$ with the rate at most $\max(1,\iota(\gamma,\alpha))$ before scaling. The saddle connections by themselves do not necessarily exist for all $t\in (t_0-\delta, t_0+\delta)$ but the total rate of change over all saddle connections is bounded by the maximum number of saddle connections and $\iota(\gamma,\alpha)$. Therefore horizontal flat grafting along a fixed simple closed curve is continuous as a function of $t$.

Suppose that a geodesic representative of $\alpha$ with respect to $q$ contains a horizontal saddle connection. We prove discontinuity by constructing a sequence $\{ q_i \}$ in $QD_1(S)$ such that $q_i \to q$ but $\overline{HG}_t(q_i,\alpha) \not \to \overline{HG}_t(q,\alpha)$. 

Consider the rotation action on the space of quadratic differentials induced by the action of $SL_2(\mathbb{R})$ on $\mathbb{R}^2$. We denote the counterclockwise rotation by an angle $\theta$ by $e^{i \theta} q$. It is clear that $e^{i\theta}q \to q$ as $\theta\to 0$. It suffices to show that $q^+$, the limit of $\overline{HG}_t(e^{i \theta} q,\alpha)$ as $\theta\to 0^+$, is not equal to $q^-$, the limit of $\overline{HG}_t(e^{i \theta} q,\alpha)$ as $\theta\to 0^-$.

Fix a geodesic representative of $\alpha$ with respect to $q$. We compare the grafting cylinders of $\alpha$ in $\overline{HG}_t(e^{i \theta} q,\alpha)$ for $\theta = \pm\delta$, $\delta>0$ a small number. As long as the angle $\delta$ is chosen such that the geodesic representative of $\alpha$ contains no horizontal saddle connections, the grafting cylinder is a sequence of parallelograms of width $t$ before scaling and an identification of the top and bottom edges of the parallelograms. We see that the limits $q^+$ and $q^-$ are well-defined and they differ in the degenerate parallelograms for each horizontal saddle connection. The different degenerated parallelograms show that the vertical first return map $f$ of $\alpha$ with respect to $q^+$ and $q^-$ differ. Hence $q^+ \neq q^-$ and $\overline{HG}_t(\ \cdot\ ,\alpha)$ is not continuous at $q$.

If no geodesic representatives of $\alpha$ contain a horizontal saddle connection with respect to $q_0$, then we consider conditions on $q$ that guarantee that $\overline{HG}_t(q,\alpha)$ will lie inside $B_{\overline{HG}_t(q_0,\alpha)}(K,\varepsilon)$. Let $\Sigma \subset \mathcal{S}(S)$ be the finite set of simple closed curves that has length less than $K$ with respect to $\overline{HG}_t(q_0,\alpha)$. If we let $K'$ be greater than the maximum length of a curve in $\Sigma$ with respect to $q_0$, then it suffices to study the finite set $\Sigma$ when working inside the open set $B_{q_0}(K',\varepsilon')$.

For any $\gamma \in \Sigma$, we know that $\ell_{\overline{HG}_t(q_0,\alpha)}(\gamma) < K$ and $\ell_{q_0}(\gamma) < K'$. By Theorem~\ref{mainpropH} again, 
\[
| \iota(\gamma,\mathcal{F}_H(q)) - \iota(\gamma,\mathcal{F}_H(q_0))| < \varepsilon \Rightarrow | \iota(\gamma,\mathcal{F}_H(\overline{HG}_t(q,\alpha))) - \iota(\gamma,\mathcal{F}_H(\overline{HG}_t(q_0,\alpha)))| < \varepsilon.
\]
The intersection number with the vertical foliation is scaled down by a normalizing factor depending on $\iota(\gamma,\mathcal{F}_H(q_0))$. The scaling is continuous since we can choose $K'$ and $\varepsilon'$ such that $\iota(\gamma,\mathcal{F}_H(q))$ is arbitrarily close to $\iota(\gamma,\mathcal{F}_H(q_0))$ for $q \in B_{q_0}(K',\varepsilon')$. 

The changes of the intersection number with the vertical foliation $\iota(\gamma,\mathcal{F}_V(\ \cdot \ ))$ under horizontal flat grafting can happen in two other ways. If a geodesic representative of $\gamma$ with respect to $q_0$ has consecutive saddle connections with opposite signs in their $x$-components, then grafting can cause cancellation depending on their relative positions to the grafting cylinder. Finally, if $\gamma$ has to cross the grafting cylinder horizontally, then the intersection number changes as if we added a horizontal saddle connection of length $t$. It suffices to prove that each of these changes vary continuously with respect to $q \in B_{q_0}(K',\varepsilon')$.

Fix geodesic representatives of $\alpha$ and $\gamma$ with respect to $q_0$. Let $\delta$ be $\varepsilon / N$ where $N$ is the maximum number of saddle connections in a geodesic representative of $\gamma$ with respect to any $q \in B_{q_0}(K',\varepsilon)$. Refine $K'$ and $\varepsilon'$ such that geodesic representatives of $\alpha$ and $\gamma$ with respect to any $q \in B_{q_0}(K',\varepsilon')$ are $\delta$-close in the Hausdorff metric, that is, saddle connections are either $\delta$-close to a saddle connection with respect to $q_0$ or $\delta$-close to the zero vector.

The new cancellations that occur for $q \mapsto \overline{HG}_t(q,\alpha)$ arises from the saddle connections that are $\delta$-close to zero. By refining $K'$ and $\varepsilon'$ further, we can maintain the cancellations that occur for $q_0 \mapsto \overline{HG}_t(q_0,\alpha)$. Hence the total changes due to cancellation is continuous with respect to $q \in B_{q_0}(K',\varepsilon')$.

The intersection number $\iota(\gamma,\alpha)$ is constant, hence we consider the orientation of the added horizontal saddle connections. If the orientation matches, then we simply see an increase in intersection number. If the orientation is opposite, cancellation will happen and that is the same as the case of consecutive saddle connections with opposite signs in their $x$-components. The gap between increase or cancellation is the reason why $\alpha$ cannot contain a horizontal saddle connection at $q_0$. The orientation is determined by the slopes of saddle connections in their geodesic representatives and examples will be given in the Appendix. Therefore once again we refine $K'$ and $\varepsilon'$ to show that $q\in B_{q_0}(K',\varepsilon')$ implies that $\overline{HG}_t(q,\alpha) \in \overline{HG}_t(q_0,\varepsilon)$.
\end{proof}

The discontinuity generated by horizontal saddle connections can be used to prove that $\overline{HG}_t(q,\ \cdot\ )$ as a function over simple closed curves cannot extend continuously to $\mathcal{MF}(S)$. Horizontal flat grafting can be defined for weighted simple multi-curves but the limit to a measured foliation might not exist. 


\section{Strata of quadratic differentials and flat grafting}\label{STRAT}

As shown in the examples in Section~\ref{EXAMP}, we see that flat grafting might map $q$ to a different stratum. The combinatorial data of a geodesic representative of $\alpha$ in $q$ determines the stratum $\overline{HG}_t(q,\alpha)$ would lie in. The first step is the necessary and sufficient conditions of geodesic representatives such that $q$ and $\overline{HG}_t(q,\alpha)$ lie in the same stratum.

\begin{prop}\label{StratumStay}
Let $q \in QD(S,\sigma)$. There exists $\delta > 0$ where $\overline{HG}_t(q,\alpha) \in QD(S,\sigma)$ for all $0 < t < \delta$ if and only if a geodesic representative of $\alpha$ with respect to $q$ satisfies the following conditions.
\begin{itemize}
\item At each marked point of cone angle at least $2\pi$, consecutive saddle connections have one-sided angle at most $2\pi$ and same sign in their $y$-components. (If $y$-component is zero, consider $e^{i\theta}q$ with $\theta\to 0^+$.)
\item At each non-marked cone point of cone angle at least $4\pi$, consecutive saddle connections have one of the two angles at most $2\pi$ and same sign in their $y$-components. (If $y$-component is zero, consider $e^{i\theta}q$ with $\theta\to 0^+$.)
\end{itemize}
\end{prop}

\begin{proof}
Recall the angle condition for geodesic representatives of a simple closed curve. Any geodesic representative is either a cylinder curve or a finite list of saddle connections. Consecutive saddle connections meet at a cone point and the angle between them is the two angles that add up to the cone angle of the point. If the cone point is a marked point, we treat the point as a puncture and the angle between the saddle connections is the angle away from the marked point, hence one-sided. The angle condition requires the angles to be at least angle $\pi$. 

We choose $\delta$ in the statement to be the length of the shortest horizontal saddle connection in $q$. In other words, if $q$ contains no horizontal saddle connection, then $\overline{HG}_t(q,\alpha)$ lie in the same stratum for all $t\geq 0$ if and only if $\alpha$ satisfies the given conditions. This follows from the observation that the strata structure can only change from either cone points splitting or merging.

First we prove that the cone point splits for small $t$ value whenever $\alpha$ does not satisfy the conditions. Note that the minimum cone angle at a cone point is $\pi$ for marked points and $3\pi$ for non-marked points. Consider the horizontal sides of the parallelograms generated by the horizontal flat grafting. For a fixed cone point, we obtain a horizontal segment for each pair of consecutive saddle connections meeting at the cone point. The total excess angle over all endpoints of the horizontal segments will be the excess angle at the cone point, where the excess angle is the cone angle minus $2\pi$. If there is a pair with both angles greater or equal to $2\pi$, then the excess angle will be split among at least two points, that is, the cone point splits. In the marked point case, because the marked point counts as a cone point regardless of cone angle, the cone point splits if a non-marked endpoint of the horizontal segments gets excess angle. The case when the $y$-components are opposite sign is similar, see Figure~\ref{ConeSplit}. The sign of the $y$-component gives a subtle bound on the angles. For example, angles that appear to be less than $\pi$ in the picture has to be at least $2\pi$.

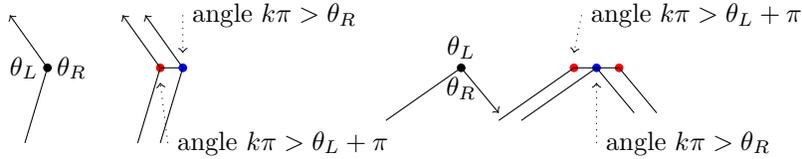
\begin{figure}[ht]
\begin{tikzpicture}

\draw[fill] (0,1) circle (0.05);
\draw[->] (-0.3,0) -- (0,1) -- (-0.5,1.7);
\draw (0,1) node[right]{$\theta_R$};
\draw (0,1) node[left]{$\theta_L$};

\draw[red,fill] (1.5,1) circle (0.05);
\draw (1.5,1) -- (1.8,1);
\draw[blue,fill] (1.8,1) circle (0.05);
\draw[->] (1.2,0) -- (1.5,1) -- (1,1.7);
\draw[->] (1.5,0) -- (1.8,1) -- (1.3,1.7);
\draw[dotted,->] (1.8,1.7) node[right]{angle $ k\pi > \theta_R$} -- (1.8,1.2);
\draw[dotted,->] (1.6,0) node[right]{angle $k\pi > \theta_L + \pi$} -- (1.5,0.8);

\draw[fill] (5.5,1) circle (0.05);
\draw[->] (4.5,0.3) -- (5.5,1) -- (6,0.4);
\draw (5.5,1) node[below]{$\theta_R$};
\draw (5.5,1) node[above]{$\theta_L$};

\draw[red,fill] (7,1) circle (0.05);
\draw[blue,fill] (7.3,1) circle (0.05);
\draw[red,fill] (7.6,1) circle (0.05);
\draw (6,0.3) -- (7,1) -- (7.3,1) -- (6.3,0.3);
\draw (7.8,0.4) -- (7.3,1) -- (7.6,1) -- (8.1,0.4);
\draw[dotted,->] (7.3,0) node[right]{angle $k\pi > \theta_R$} -- (7.3,0.8);
\draw[dotted,->] (7.1,1.7) node[right]{angle $k\pi > \theta_L + \pi$} -- (7,1.2);

\end{tikzpicture}
\caption{The local picture of cone points splitting under flat grafting.}
\label{ConeSplit}
\end{figure}

If $\alpha$ satisfies the conditions, then the total excess angle is concentrated at a single cone point. Therefore $q$ and $\overline{HG}_t(q,\alpha)$ are in the same stratum for all $t < \delta$ due to having no other cone point splitting or merging possible.
\end{proof}

In a different viewpoint, we observe that the stratum structure stabilizes as $t$ goes to infinity.

\begin{prop}\label{LongTerm}
For all $q\in QD_1(S)$ and $\alpha\in \mathcal{S}(S)$, there exists $K$ such that $\overline{HG}_t(q,\alpha)$ lies in the same stratum for all $t > K$.
\end{prop}

\begin{proof}
In fact, we can classify the exact sequence of strata structure of $\overline{HG}_t(q,\alpha)$. If $\alpha$ is a cylinder curve, then $\overline{HG}_t(q,\alpha)$ lies in the same stratum for all $t \geq 0$. Otherwise, consider the finite set of horizontal saddle connections in the unique geodesic representative of $\alpha$.

If the geodesic representative of $\alpha$ contains no horizontal saddle connections, then $\overline{HG}_t(q,\alpha)$ lies in the same stratum for all $t > 0$. The stratum is the result of distributing total excess angle at each cone point $\alpha$ passes through. If there is at least one horizontal saddle connection in $\alpha$, then the horizontal segments discussed in the proof of Proposition~\ref{StratumStay} will overlap at some point, see Figure~\ref{Overlap}.

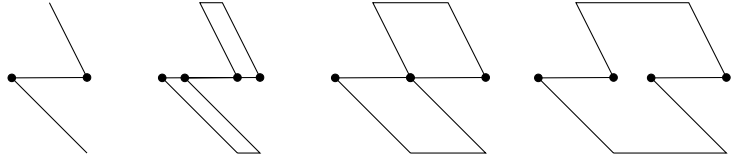
\begin{figure}[ht]
\begin{tikzpicture}

\draw[fill] (1,1) circle (0.05);
\draw[fill] (2,1.005) circle (0.05);
\draw (2,0) -- (1,1) -- (2,1.005) -- (1.5,2);

\draw[fill] (3.3,1) circle (0.05);
\draw[fill] (4.3,1.005) circle (0.05);
\draw[fill] (3,1) circle (0.05);
\draw[fill] (4,1.005) circle (0.05);
\draw (4,0) -- (4.3,0) -- (3.3,1) -- (4.3,1.005) -- (3.8,2) -- (3.5,2) -- (4,1.005) -- (3,1) -- (4,0);

\draw[fill] (6.3,1) circle (0.05);
\draw[fill] (7.3,1.005) circle (0.05);
\draw[fill] (5.3,1) circle (0.05);
\draw[fill] (6.3,1.005) circle (0.05);
\draw (6.3,0) -- (7.3,0) -- (6.3,1) -- (7.3,1.005) -- (6.8,2) -- (5.8,2) -- (6.3,1.005) -- (5.3,1) -- (6.3,0);

\draw[fill] (9.5,1) circle (0.05);
\draw[fill] (10.5,1.005) circle (0.05);
\draw[fill] (8,1) circle (0.05);
\draw[fill] (9,1.005) circle (0.05);
\draw (9,0) -- (10.5,0) -- (9.5,1) -- (10.5,1.005) -- (10,2) -- (8.5,2) -- (9,1.005) -- (8,1) -- (9,0);

\end{tikzpicture}
\caption{Cone points merging then splitting.}
\label{Overlap}
\end{figure}

The change in strata structure shown in Figure~\ref{Overlap} is the case when the horizontal saddle connection has multiplicity one in $\alpha$. If a horizontal saddle connection has length $k$ with multiplicity $m$, then the merging and splitting happens at $t = k/m, 2k/m, \ldots, k$. Combining this fact with the distribution of excess angles, we can use the combinatorial data of the geodesic representative to determine the strata structure of $\overline{HG}_t(q,\alpha)$. Also as a conclusion, we can take $K$ to be the length of the longest horizontal saddle connection in $\alpha$.
\end{proof}

The changes in strata structure under a small flat grafting correspond to decreasing multiplicity of a root of a polynomial. The case of special interest is when the resulting polynomial has only simple roots. In our case, the principal stratum is the stratum with $\sigma_p = (3\pi,\ldots, 3\pi, \pi,\ldots, \pi;-1)$. We prove Theorem~\ref{principal} in an indirect way.

Recall that a slice of the space of unit-area quadratic differentials is the set
\[
QD_1^\mu(S) = \{ q \in QD_1(S) \mid \mu = \mathcal{F}_H(q)\}.
\]
Since any horizontal flat grafting preserves the horizontal foliation, $\overline{HG}_t(\ \cdot \ ,\alpha)$ maps $QD_1^\mu(S)$ to itself. We consider the image of the map $\overline{HG}_\cdot (q, \ \cdot \ )$ over $\mathbb{R}_{\geq 0} \times \mathcal{S}(S)$. 

\begin{theorem}\label{Dense}
The image of $\overline{HG}_\cdot (q, \ \cdot \ )$ is dense in $QD_1^\mu(S)$ for all $q \in QD_1^\mu(S)$.
\end{theorem}

\begin{proof}
The topology on $QD_1^\mu(S)$ is the subspace topology. Consider an open set $B = B_{q_0}(K,\varepsilon) \cap QD_1^\mu(S)$. By the denseness of weighted simple closed curves in the space of measured foliations, there exists a quadratic differential $q_\alpha \in B$ for some $\alpha \in \mathcal{S}(S)$, where $\mathcal{F}_V(q_\alpha)$ is equal to $(\iota(\alpha,\mu))^{-1}\cdot \alpha$, a weighted simple closed curve. Let $B' = B_{q_\alpha}(K',\varepsilon') $ be a subset of $B$. We will show that $\overline{HG}_t(q,\alpha)$ lies inside $B'$ for $t$ large enough. We make a quick observation on $\overline{HG}_t(q,\alpha)$ for $t$ large.

\begin{lemma}\label{EventualCyl}
If $\alpha$ has a non-horizontal saddle connection with respect to $q$, then $\alpha$ is a cylinder curve with respect to $\overline{HG}_t(q,\alpha)$ for some $t>0$.
\end{lemma}

\begin{proof}
Consider the grafting cylinder of the non-normalized $HG_t(q,\alpha)$ for large $t$. Every non-degenerate parallelogram in the grafting cylinder contains a maximal rectangle. The rectangle degenerates to a horizontal line segment for horizontal saddle connections. If $t$ is greater than the total horizontal span of $\alpha$, $\iota(\alpha,\mathcal{F}_V(q))$, then the maximal rectangles will stack up in the sense that there is a vertical line passing through them all. For $t$ large enough there is a cylinder curve inside the union of the rectangles that is homotopic to $\alpha$.
\end{proof}

Horizontal flat grafting along a cylinder curve is very simple. The cylinder as a parallelogram expands horizontally in $HG_t(q,\alpha)$ and converges to a rectangle in $\overline{HG}_t(q,\alpha)$. The next lemma helps us reduce to that case.

\begin{lemma}
We have $HG_{t_1}(HG_{t_2}(q,\alpha),\alpha) = HG_{t_1+t_2}(q,\alpha)$ and there exists a function $f$ that depends only on $I := \iota(\alpha,\mathcal{F}_H(q))$ such that
\[
\overline{HG}_{t_1}\left(\overline{HG}_{t_2}(q,\alpha),\alpha\right) = \overline{HG}_{f(t_1,t_2)}(q,\alpha).
\]
\end{lemma}

\begin{proof}
The first part follows from the fact that the geodesic representative of $\alpha$ with respect to the result of horizontal flat grafting will stay inside the grafting cylinder. It will always traverse each parallelogram exactly once from top to bottom. Hence the parallelograms expand horizontally at the same rate even when the geodesic representative changed.

The second part is a straightforward computation. The width of the grafting cylinder is 
\[
\frac{f(t_1,t_2)}{1+I\cdot f(t_1,t_2)} = \frac{\frac{t_2}{1+I\cdot t_2} + t_1}{1+I\cdot t_1}.
\]
Therefore $f(t_1,t_2) = t_2 + t_1 (1 + I\cdot t_2).$
\end{proof}

Now we return to the proof of Theorem~\ref{Dense}. We want to show that for $t$ large enough, $\overline{HG}_t(q,\alpha)$ lies in $B_{q_{\alpha}}(K',\varepsilon)$. Let $\gamma$ be a simple closed curve satisfying $\ell_{q_\alpha}(\gamma) < K'$. The horizontal measured foliation is fixed so we will focus on comparing $\iota(\gamma, \mathcal{F}_V(q_\alpha))$ and $\iota(\gamma, \mathcal{F}_V(\overline{HG}_t(q,\alpha)))$.

The intersection number $\iota(\gamma, \mathcal{F}_V(q_\alpha))$ is equal to $(\iota(\alpha,\mu))^{-1}\iota(\gamma, \alpha)$ by our choice of $q_\alpha$. The simple closed curve $\alpha$ contains a non-horizontal saddle connection with respect to any quadratic differential in $QD^\mu_1(S)$ by the choice of $q_\alpha$ as well. 

Consider a geodesic representative of $\gamma$ with respect to $\overline{HG}_t(q,\alpha)$ where $\alpha$ is a cylinder curve with respect to $\overline{HG}_t(q,\alpha)$ by Lemma~\ref{EventualCyl}. As $t$ goes to infinity, the saddle connections of a geodesic representative of $\gamma$ are either becoming more and more vertical or it intersection the cylinder $C_\alpha$. The intersection number $\iota(\gamma, \mathcal{F}_V(\overline{HG}_t(q,\alpha)))$ converges to $(\iota(\alpha,\mu))^{-1}\iota(\gamma, \alpha)$ and the horizontal flat grafting ray eventually lies inside $B'$.
\end{proof}

As a corollary we can almost construct paths between any pair of quadratic differentials. We can get as close as we want as long as we add in a stretch/shrink factor for the horizontal and vertical measured foliation.

\begin{cor}
For all $q,q' \in QD_1(S), K>0, \varepsilon>0$, there exists $x \in \mathbb{R}, t_H, t_V > 0, \alpha_H, \alpha_V \in \mathcal{S}(S)$ such that 
\[
{{e^x \ \ 0\ }\choose{\ 0 \ \  e^{-x}}} \cdot \overline{HG}_{t_H}\left(\overline{VG}_{t_V}(q',\alpha_V),\alpha_H\right) \in B_q(K,\varepsilon).
\]
\end{cor}

We have shown that for all $q$, there exists $t>0$ and $\alpha \in\mathcal{S}(S)$ such that $\overline{HG}_t(q,\alpha) \in QD_1(S,\sigma_p)$. As a corollary, we obtain Theorem~\ref{principal} if we just choose $\alpha$ such that it contains no horizontal saddle connections with respect to $q$. This observation comes from our analysis of strata structure changes from Proposition~\ref{LongTerm} and the denseness of valid choices of $\alpha$. Making use of the principal stratum we can construct different paths between quadratic differentials via deforming both toward a fixed quadratic differential in the principal stratum. The period coordinate in the principal stratum allows us to connect the two deforming paths once they fall in a fixed open neighborhood.


\section{Length rigidity on slices}\label{RIGID}

In \cites{BanLei,DLR,Fu}, the authors considered the marked length rigidity problem for $\mathsf{Flat}(S)$. The length rigidity does not apply to quadratic differentials since a circle of quadratic differentials induce the same metric in $\mathsf{Flat}(S)$. However, each quadratic differential in a slice $QD^\mu(S)$ induces a unique metric and the length rigidity problem can be posed.

Consider the special case when $\mu$ is a weighted simple closed curve $\omega\cdot \alpha$. This is a special case that motivated our work. Recall that a family of metrics is length rigid with respect to a set of closed curves if the length vector function is injective.

\begin{theorem}\label{9g-9}
The set $QD_1^{\omega\cdot\alpha}(S)$ for $\omega > 0$ and $\alpha\in\mathcal{S}(S)$ is length rigid with respect to a set of $(9g+3p-10)$ simple closed curves.
\end{theorem}

\begin{proof}
By definition, $\mathcal{F}_H(q) = \omega\cdot\alpha$ for all $q\in QD_1^{\omega\cdot\alpha}(S)$. If $\gamma \in \mathcal{S}(S)$ is disjoint from $\alpha$, then $\gamma$ is horizontal with respect to any $q\in QD_1^{\omega\cdot\alpha}(S)$. When a closed curve is horizontal with respect to a quadratic differential, the length of the closed curve is given by the intersection number with the vertical measured foliation. With a similar wording, we prove the intersection rigidity of measured foliations with respect to a finite set of simple closed curves.

\begin{lemma}\label{MFrigid}
There exists a set $\Sigma$ of $(9g+3p-9)$ simple closed curves such that the function
\[
\iota(\Sigma, \ \cdot \ ) : \mathcal{MF}(S) \to \mathbb{R}^{9g+3p-9}
\]
is injective.
\end{lemma}

\begin{proof}
The choice of $\Sigma$ is analogous to the $(9g-9)$-theorem in \cite{FaMa}. Take a maximal set of disjoint simple closed curves, called pants curves. The pants curve set contains $(3g+p-3)$ curves. The intersection number with respect to the pants curves determines a measured foliation up to twisting along pants curves. Consider a set of transverse curves where each curve intersects exactly one pants curve and intersects it minimally (once or twice). The relationship between the twisting along a pants curve and the intersection number with respect to the transverse curve is given by an absolute value function. The slope of the relationship is fixed hence it suffices to pick a pair of transverse curve for each pants curve. Therefore the intersection number with the constructed set $\Sigma$ consisting of $(9g+3p-9)$ curves determine the measured foliation uniquely.
\end{proof}

The length of $\alpha$ with respect to any $q\in QD_1^{\omega\cdot\alpha}(S)$ is equal to $1/\omega$ since the quadratic differential has unit-area. By picking a set of pants curves that include $\alpha$, the length vector of the pants curves is equal to the intersection vector of the intersections number with the vertical measured foliation. The same is true for any transverse curve corresponding to a pants curve that is not $\alpha$. 

The lemma above shows that the length vector of a set of $(9g+3p-12)$ simple closed curves determine the vertical measured foliation up to twisting along $\alpha$. In other words, the preimage of a length vector of length $(9g+3p-12)$ is a family of quadratic differentials obtained by the shearing action ${{1 \ s}\choose{0\ 1}}$. The length of a non-horizontal simple closed curve is a convex function with respect to $s$, hence it suffices to choose two simple closed curves to determine $s$, meaning that $QD_1^{\omega\cdot\alpha}(S)$ is length rigid with respect to a set of $(9g+3p-10)$ curves.
\end{proof}

The above result is special in the sense that the proof does not extend to even the case of multicurves. For the general case we prove a local result instead. Theorem~\ref{rigidity} is a corollary of the following theorem. 

\begin{theorem}\label{FiniteClosed}
For all $q \in QD_1^\mu(S)$, there exists $K>0$ and $\varepsilon>0$ such that $B_q(K,\varepsilon)$ is length rigid with respect to a finite set of closed curves.
\end{theorem}

\begin{proof}
The idea is to begin with a set $\Sigma$ of simple closed curves as in Lemma~\ref{MFrigid} above. While the intersection vector of the vertical foliation with $\Sigma$ determines the quadratic differential in $QD_1^\mu(S)$, the length vector $\ell_\cdot(\Sigma)$ is not necessarily determined by the intersection vector. However, the lengths of some of the saddle connections in the geodesic representatives of $\Sigma$ are changing as we vary the quadratic differential. We can construct a set of closed curves using the idea in the proof of \cite[Theorem~3]{Fu} such that local rigidity is obtained. The following argument is for the case when $S$ has no marked points. The case with marked points is analogue to the adjustments in the proof of \cite[Theorem~3]{Fu}.

If $q$ lies in the principal stratum, then \cite[Theorem~3]{Fu} proves the local rigidity result. If $q$ is not in the principal stratum, then we need to take into account the splitting of cone points. An important observation is that any splitting of the cone points can be characterized by the horizontal saddle connections that arise from the splitting. In Figure~\ref{SClength} we illustrate the method to determine the length of a saddle connection via lengths of closed curves.

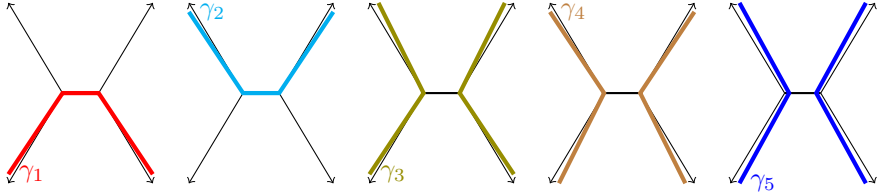
\begin{figure}[ht]
\begin{tikzpicture}[scale=1.2]

\draw[thick] (-4.2,0) -- (-3.8,0); 
 \draw [<-] (-4.8,1) -- (-4.2,0); 
 \draw [<-] (-4.8,-1) -- (-4.2,0); 
 \draw [->] (-3.8,0) -- (-3.2,1); 
 \draw [->] (-3.8,0) -- (-3.2,-1); 

\draw[thick] (-2.2,0) -- (-1.8,0); 
 \draw [<-] (-2.8,1) -- (-2.2,0); 
 \draw [<-] (-2.8,-1) -- (-2.2,0); 
 \draw [->] (-1.8,0) -- (-1.2,1); 
 \draw [->] (-1.8,0) -- (-1.2,-1); 

\draw[thick] (-0.2,0) -- (0.2,0); 
 \draw [<-] (-0.8,1) -- (-0.2,0); 
 \draw [<-] (-0.8,-1) -- (-0.2,0); 
 \draw [->] (0.2,0) -- (0.8,1); 
 \draw [->] (0.2,0) -- (0.8,-1); 

\draw[thick] (1.8,0) -- (2.2,0); 
 \draw [<-] (1.2,1) -- (1.8,0); 
 \draw [<-] (1.2,-1) -- (1.8,0); 
 \draw [->] (2.2,0) -- (2.8,1); 
 \draw [->] (2.2,0) -- (2.8,-1); 

\draw[thick] (3.8,0) -- (4.2,0); 
 \draw [<-] (3.2,1) -- (3.8,0); 
 \draw [<-] (3.2,-1) -- (3.8,0); 
 \draw [->] (4.2,0) -- (4.8,1); 
 \draw [->] (4.2,0) -- (4.8,-1); 

\draw[red,ultra thick] (-4.8,-0.9) node[right]{$\gamma_1$} -- (-4.2,0) -- (-3.8,0) -- (-3.2,-0.9);
\draw[cyan,ultra thick] (-2.8,0.9) node[right]{$\gamma_2$} -- (-2.2,0) -- (-1.8,0) -- (-1.2,0.9);
\draw[blue,ultra thick] (3.3,1) -- (3.85,0) -- (3.3,-1) node[right]{$\gamma_5$};
\draw[blue,ultra thick] (4.7,-1) -- (4.15,0) -- (4.7,1);
\draw[olive,ultra thick] (-0.8,-0.9) node[right]{$\gamma_3$} -- (-0.2,0) -- (-0.7,1);
\draw[olive,ultra thick] (0.8,-0.9) -- (0.2,0) -- (0.7,1);
\draw[brown,ultra thick] (1.3,-1) -- (1.8,0) -- (1.2,0.9) node[right]{$\gamma_4$};
\draw[brown,ultra thick] (2.7,-1) -- (2.2,0) -- (2.8,0.9);

\end{tikzpicture}
\caption{Length of the saddle connection is a linear combination of $\ell(\gamma_j)$'s.}
\label{SClength}
\end{figure}

Figure~\ref{SClength} shows five geodesics near a pair of cone points connected by a horizontal saddle connection. The underlying black lines illustrate the horizontal direction and the cone points in Figure~\ref{SClength} both have cone angle $3\pi$. In a neighborhood of the quadratic differential that admits a similar structure, the length of the saddle connection is equal to 
\[
\frac{1}{2}[\ell(\gamma_1)+\ell(\gamma_2)-\ell(\gamma_3)-\ell(\gamma_4)+\ell(\gamma_5)].
\]
For any given $q \in QD_1^\mu(S)$, we construct five closed geodesics for every four distinct horizontal directions at a cone point. In other words, a finite set of closed curves satisfy the property that their length vector function over $B_q(K,\varepsilon) \cap QD_1^\mu(S)$ detects cone point splitting.

Once we obtain control over cone points splitting, it suffices to choose $\Sigma$ satisfying Lemma~\ref{MFrigid} and construct closed curves for each saddle connection of geodesic representatives of $\Sigma$ with respect to $q$. By the action of the mapping class group we can choose $\Sigma$ to be non-vertical to avoid the edge cases. For saddle connections between cone points of cone angle $3\pi$, the construction of $\gamma_j$'s is the same as \cite[Theorem~3]{Fu}. For saddle connections between cone points of higher cone angles, the construction needs to respect the possibility of cone point splitting. However, with the extra cone angle, the construction has a lot of extra room for choosing the directions leaving the cone points such that the length of the saddle connection is still equal to the same linear combination in $B_q(K,\varepsilon) \cap QD_1^\mu(S)$. Therefore the combination of the closed curves for splitting and saddle connections of $\Sigma$ is the desired finite set where $B_q(K,\varepsilon)$ is chosen small enough to maintain the linear combination equations.
\end{proof}

The gap between Theorem~\ref{9g-9} and Theorem~\ref{FiniteClosed} in this section prompts further research into the relationship between the length spectrum and the slices of quadratic differentials.

\begin{bibdiv}
\begin{biblist}

\bib{AEZ}{article}{
   author={Athreya, Jayadev S.},
   author={Eskin, Alex},
   author={Zorich, Anton},
   title={Right-angled billiards and volumes of moduli spaces of quadratic
   differentials on $\Bbb C\rm P^1$},
   language={English, with English and French summaries},
   note={With an appendix by Jon Chaika},
   journal={Ann. Sci. \'Ec. Norm. Sup\'er. (4)},
   volume={49},
   date={2016},
   number={6},
   pages={1311--1386},
   issn={0012-9593},
}

\bib{BanLei}{article}{
   author={Bankovic, Anja},
   author={Leininger, Christopher J.},
   title={Marked-length-spectral rigidity for flat metrics},
   journal={Trans. Amer. Math. Soc.},
   volume={370},
   date={2018},
   number={3},
   pages={1867--1884},
   issn={0002-9947},
}
\bib{Boi}{article}{
   author={Boissy, Corentin},
   title={Degenerations of quadratic differentials on $\Bbb C\Bbb P^1$},
   journal={Geom. Topol.},
   volume={12},
   date={2008},
   number={3},
   pages={1345--1386},
   issn={1465-3060},
}

\bib{Bon}{article}{
   author={Bonahon, Francis},
   title={Earthquakes on Riemann surfaces and on measured geodesic
   laminations},
   journal={Trans. Amer. Math. Soc.},
   volume={330},
   date={1992},
   number={1},
   pages={69--95},
   issn={0002-9947},
}

\bib{BonOta}{article}{
   author={Bonahon, Francis},
   author={Otal, Jean-Pierre},
   title={Laminations measur\'ees de plissage des vari\'et\'es hyperboliques de
   dimension 3},
   language={French, with English summary},
   journal={Ann. of Math. (2)},
   volume={160},
   date={2004},
   number={3},
   pages={1013--1055},
   issn={0003-486X},
}

\bib{DLR}{article}{
   author={Duchin, Moon},
   author={Leininger, Christopher J.},
   author={Rafi, Kasra},
   title={Length spectra and degeneration of flat metrics},
   journal={Invent. Math.},
   volume={182},
   date={2010},
   number={2},
   pages={231--277},
   issn={0020-9910},
   doi={10.1007/s00222-010-0262-y},
}

\bib{DumWol}{article}{
   author={Dumas, David},
   author={Wolf, Michael},
   title={Projective structures, grafting and measured laminations},
   journal={Geom. Topol.},
   volume={12},
   date={2008},
   number={1},
   pages={351--386},
   issn={1465-3060},
}

\bib{EMZ}{article}{
   author={Eskin, Alex},
   author={Masur, Howard},
   author={Zorich, Anton},
   title={Moduli spaces of abelian differentials: the principal boundary,
   counting problems, and the Siegel-Veech constants},
   journal={Publ. Math. Inst. Hautes \'Etudes Sci.},
   number={97},
   date={2003},
   pages={61--179},
   issn={0073-8301},
}

\bib{EMM}{article}{
   author={Eskin, Alex},
   author={Mirzakhani, Maryam},
   author={Mohammadi, Amir},
   title={Isolation, equidistribution, and orbit closures for the ${\rm
   SL}(2,\Bbb R)$ action on moduli space},
   journal={Ann. of Math. (2)},
   volume={182},
   date={2015},
   number={2},
   pages={673--721},
   issn={0003-486X},
}

\bib{FaMa}{book}{
   author={Farb, Benson},
   author={Margalit, Dan},
   title={A primer on mapping class groups},
   series={Princeton Mathematical Series},
   volume={49},
   publisher={Princeton University Press, Princeton, NJ},
   date={2012},
   pages={xiv+472},
   isbn={978-0-691-14794-9},
}

\bib{FLP}{collection}{
   author={Fathi, Albert},
   author={Laudenbach, Fran{\c{c}}ois},
   author={Po{\'e}naru, Valentin},
   title={Thurston's work on surfaces},
   series={Mathematical Notes},
   volume={48},
   note={Translated from the 1979 French original by Djun M. Kim and Dan
   Margalit},
   publisher={Princeton University Press, Princeton, NJ},
   date={2012},
   pages={xvi+254},
   isbn={978-0-691-14735-2},
}

\bib{Fu}{article}{
   author={Fu, Ser-Wei},
   title={Length spectra and strata of flat metrics},
   journal={Geom. Dedicata},
   volume={173},
   date={2014},
   pages={281--298},
   issn={0046-5755},
   doi={10.1007/s10711-013-9942-2},
}

\bib{GarMas}{article}{
   author={Gardiner, Frederick P.},
   author={Masur, Howard},
   title={Extremal length geometry of Teichm\"uller space},
   journal={Complex Variables Theory Appl.},
   volume={16},
   date={1991},
   number={2-3},
   pages={209--237},
   issn={0278-1077},
}

\bib{HubMas}{article}{
   author={Hubbard, John},
   author={Masur, Howard},
   title={Quadratic differentials and foliations},
   journal={Acta Math.},
   volume={142},
   date={1979},
   number={3-4},
   pages={221--274},
   issn={0001-5962},
}

\bib{KamTan}{article}{
   author={Kamishima, Yoshinobu},
   author={Tan, Ser P.},
   title={Deformation spaces on geometric structures},
   conference={
      title={Aspects of low-dimensional manifolds},
   },
   book={
      series={Adv. Stud. Pure Math.},
      volume={20},
      publisher={Kinokuniya, Tokyo},
   },
   date={1992},
   pages={263--299},
}

\bib{Ker}{article}{
   author={Kerckhoff, Steven P.},
   title={The Nielsen realization problem},
   journal={Ann. of Math. (2)},
   volume={117},
   date={1983},
   number={2},
   pages={235--265},
   issn={0003-486X},
}

\bib{Mas}{article}{
   author={Masur, Howard},
   title={Interval exchange transformations and measured foliations},
   journal={Ann. of Math. (2)},
   volume={115},
   date={1982},
   number={1},
   pages={169--200},
   issn={0003-486X},
}

\bib{Masur2}{article}{
   author={Masur, Howard},
   title={Closed trajectories for quadratic differentials with an
   application to billiards},
   journal={Duke Math. J.},
   volume={53},
   date={1986},
   number={2},
   pages={307--314},
   issn={0012-7094},
}

\bib{MasTab}{article}{
   author={Masur, Howard},
   author={Tabachnikov, Serge},
   title={Rational billiards and flat structures},
   conference={
      title={Handbook of dynamical systems, Vol.\ 1A},
   },
   book={
      publisher={North-Holland, Amsterdam},
   },
   date={2002},
   pages={1015--1089},
}

\bib{MZ}{article}{
   author={Masur, Howard},
   author={Zorich, Anton},
   title={Multiple saddle connections on flat surfaces and the principal
   boundary of the moduli spaces of quadratic differentials},
   journal={Geom. Funct. Anal.},
   volume={18},
   date={2008},
   number={3},
   pages={919--987},
   issn={1016-443X},
   doi={10.1007/s00039-008-0678-3},
}

\bib{ScaWol}{article}{
   author={Scannell, Kevin P.},
   author={Wolf, Michael},
   title={The grafting map of Teichm\"uller space},
   journal={J. Amer. Math. Soc.},
   volume={15},
   date={2002},
   number={4},
   pages={893--927},
   issn={0894-0347},
}

\bib{SmiWei}{article}{
   author={Smillie, John},
   author={Weiss, Barak},
   title={Finiteness results for flat surfaces: a survey and problem list},
   conference={
      title={Partially hyperbolic dynamics, laminations, and Teichm\"uller
      flow},
   },
   book={
      series={Fields Inst. Commun.},
      volume={51},
      publisher={Amer. Math. Soc., Providence, RI},
   },
   date={2007},
   pages={125--137},
}

\bib{Str}{book}{
   author={Strebel, Kurt},
   title={Quadratic differentials},
   series={Ergebnisse der Mathematik und ihrer Grenzgebiete (3) [Results in
   Mathematics and Related Areas (3)]},
   volume={5},
   publisher={Springer-Verlag, Berlin},
   date={1984},
   pages={xii+184},
   isbn={3-540-13035-7},
}

\bib{Thu2}{article}{
   author={Thurston, W. P.},
   title={The Geometry and Topology of Three-manifolds},
   journal={Princeton Lecture Notes},
   date={1980},
}

\bib{Thu}{article}{
   author={Thurston, W. P.},
   title={Earthquakes in 2-dimensional hyperbolic geometry [MR0903860]},
   conference={
      title={Fundamentals of hyperbolic geometry: selected expositions},
   },
   book={
      series={London Math. Soc. Lecture Note Ser.},
      volume={328},
      publisher={Cambridge Univ. Press, Cambridge},
   },
   date={2006},
   pages={267--289},
}

\bib{Wri}{article}{
   author={Wright, Alex},
   title={Cylinder deformations in orbit closures of translation surfaces},
   journal={Geom. Topol.},
   volume={19},
   date={2015},
   number={1},
   pages={413--438},
   issn={1465-3060},
   review={\MR{3318755}},
}

\bib{Zor}{article}{
   author={Zorich, Anton},
   title={Flat surfaces},
   conference={
      title={Frontiers in number theory, physics, and geometry. I},
   },
   book={
      publisher={Springer, Berlin},
   },
   date={2006},
   pages={437--583},
}

\end{biblist}
\end{bibdiv}
\end{document}